\documentclass[12pt]{amsart}
\usepackage{amsthm, amssymb}
\usepackage{amsfonts, amscd, color}
\usepackage{epsfig,multicol}

\theoremstyle{plain} \numberwithin{equation}{section}
\newtheorem{theo}{Theorem}[section]
\newtheorem{coro}[theo]{Corollary}
\newtheorem{prop}[theo]{Proposition}
\newtheorem{lemm}[theo]{Lemma}

\theoremstyle{definition}
\newtheorem*{defi}{Definition}
\newtheorem*{exam}{Example}
\newtheorem*{rema}{Remark}

\def\Z{\mathbb Z}
\def\C{\mathbb C}
\def\R{\mathbb R}

\def\MS{\mathcal S}
\def\MR{\mathcal R}
\def\MD{\mathcal D}
\def\MX{\mathcal X}
\def\MI{\mathcal I}

\def\Si{\Sigma}
\def\De{\Delta}

\def\b{b}

\def\v{v}
\def\la{\lambda}
\def\ch{\chi}

\def\uC{\underline{\C}}
\def\RZR{\C\times\Z}
\def\RZRn{\C^n\times\Z^n}

\def\1{\mathbf 1}
\def\0{\mathbf 0}
\def\SiX{\Si(X)}
\def\RX{\R X}

\DeclareMathOperator{\GL}{GL}

\DeclareMathOperator{\Lie}{Lie}

\DeclareMathOperator{\Ker}{Ker}
\DeclareMathOperator{\Hom}{Hom}
\DeclareMathOperator{\Int}{Int}

\begin{document}
\title{Topological toric manifolds}
\author[H. Ishida]{Hiroaki Ishida} \author[Y. Fukukawa]{Yukiko Fukukawa}  \author[M. Masuda]{Mikiya Masuda}
\address{Osaka City University Advanced Mathematical Institute, Sumiyoshi-ku, Osaka 558-8585, Japan.}
\email{ishida@sci.osaka-cu.ac.jp}

\address{Department of Mathematics, Osaka City University, Sumiyoshi-ku, Osaka 558-8585, Japan.}
\email{yukiko.fukukawa@gmail.com} \email{masuda@sci.osaka-cu.ac.jp}

\date{\today}
\thanks{The first and second authors were supported by JSPS Research Fellowships for Young Scientists}
\thanks{The third author was partially supported by Grant-in-Aid for Scientific Research 22540094}

\begin{abstract}
We introduce the notion of a topological toric manifold and a topological fan and show that there is a bijection between omnioriented topological toric manifolds and complete non-singular topological fans.  A topological toric manifold is a topological analogue of a toric manifold and the family of topological toric manifolds is much larger than that of toric manifolds.  A topological fan is a combinatorial object generalizing the notion of a simplicial fan in toric geometry. 

Prior to this paper, two topological analogues of a toric manifold have been introduced.  One is a quasitoric manifold and the other is a torus manifold.  One major difference between the previous notions and topological toric manifolds is that the former support a smooth action of an $S^1$-torus while the latter support a smooth action of a $\C^*$-torus.  We also discuss their relation in details.     
\end{abstract}

\maketitle 

\section{Introduction}

Toric geometry was established in the 1970's following the pioneering work of Demazure (see \cite{co-li-sc}, \cite{fult93}, \cite{oda88} for the history).  It provides examples of explicit algebraic varieties and finds many interesting connections with combinatorics.  A toric variety of complex dimension $n$ is a normal algebraic variety over the complex numbers $\C$ with an effective algebraic action of $(\C^*)^n$ having an open dense orbit, where $\C^*=\C\backslash\{0\}$.  On the other hand, a fan of real dimension $n$ is a collection of cones in $\R^n$ with the origin as vertex satisfying certain conditions.  A fundamental theorem in toric geometry says that there is a bijective correspondence between toric varieties of complex dimension $n$ and fans of real dimension $n$.  

Among toric varieties, compact non-singular toric varieties, which we call \emph{toric manifolds} in this paper, are well studied.  In particular, their cohomology rings, Chern classes and genera such as signature are explicitly described in terms of the associated fans.  Among toric manifolds, projective ones provide examples of symplectic manifolds.  A projective toric manifold $X$ admits a moment map whose image is a simple convex polytope and can be identified with the orbit space of $X$ by the compact torus $(S^1)^n$ of $(\C^*)^n$.   

Previous to this paper, two topological analogues of a toric manifold have been introduced and a theory similar to toric geometry is developed for them using topological technique.  One topological analogue is what is now called a \emph{quasitoric manifold}\footnote{Davis-Januszkiewicz \cite{da-ja91} uses the terminology \emph{toric manifold} but it was already used in algebraic geometry as the meaning of smooth toric variety, so Buchstaber-Panov \cite{bu-pa02} started using the word \emph{quasitoric manifold}.} introduced by Davis-Januszkiewicz \cite{da-ja91} around 1990 and the other is a \emph{torus manifold} introduced by Masuda \cite{masu99}\footnote{In \cite{masu99}, the notion of \emph{unitary toric manifold} is introduced.  It is a torus manifold with invariant unitary (or weakly almost complex) structure.} and Hattori-Masuda \cite{ha-ma03} around 2000.  

A quasitoric manifold is a closed smooth manifold $M$ of even dimension, say $2n$, with an effective smooth action of $(S^1)^n$, such that $M$ is locally equivariantly diffeomorphic to a representation space of $(S^1)^n$ and the orbit space $M/(S^1)^n$ is a simple convex polytope.  A projective toric manifold with the restricted action of the compact torus is a quasitoric manifold but there are many quasitoric manifolds which do not arise this way.  For example $\C P^2\# \C P^2$ with a smooth action of $(S^1)^2$ is quasitoric but not toric because it does not allow a complex (even almost complex) structure, as is well-known.  Davis-Januszkiewicz \cite{da-ja91} show that quasitoric manifolds $M$ are classified in terms of pairs $(Q,v)$ where $Q$ is a simple convex polytope identified with the orbit space $M/(S^1)^n$ and $v$ is a function on the facets of $Q$ with values in $\Z^n$ satisfying a certain unimodularity condition.  

A torus manifold is a closed smooth manifold $M$ of even dimension, say $2n$, with an effective smooth action of $(S^1)^n$ having a fixed point.  An orientation datum called an omniorientation is often incorporated in the definition of a torus manifold.  The action of $(S^1)^n$ on a toric or quasitoric manifold has a fixed point, so they are torus manifolds.  A typical and simple example of a torus manifold which is neither toric nor quasitoric is $2n$-sphere $S^{2n}$ with a natural smooth action of $(S^1)^n$ for $n\ge 2$.  The orbit space $S^{2n}/(S^1)^n$ is contractible but there are many torus manifolds whose orbit spaces by the torus action are not contractible unlike in the case of toric or quasitoric manifolds.  Although the family of torus manifolds is much larger than that of toric or quasitoric manifolds, one can associate a combinatorial object $\De(M)$ called a multi-fan to an omnioriented torus manifold $M$.  Roughly speaking, a multi-fan is also a collection of cones but cones may overlap unlike ordinary fans.  When $M$ arises from a toric manifold, the multi-fan $\De(M)$ agrees with the ordinary fan associated with $M$.  In general, the multi-fan $\De(M)$ does not determine $M$, but it contains a lot of geometrical information on $M$, e.g. genera of $M$ such as signature, Hirzebruch $T_y$ (or $\chi_y$) genus   and elliptic genus can be described in terms of $\De(M)$  (\cite{ha-ma03}, \cite{ha-ma05}, \cite{masu99}, a similar description can be found in \cite{pano99} and \cite{pano01} for the Hirzebruch genus). 
 
In this paper, we introduce a third topological analogue of a toric manifold, which we believe is the correct topological analogue.  Remember that a toric manifold of complex dimension $n$ is a compact smooth algebraic variety with an effective algebraic action of $(\C^*)^n$ having an open dense orbit.  It is known that a toric manifold is covered by finitely many invariant open subsets each equivariantly and algebraically isomorphic to a direct sum of complex one-dimensional \emph{algebraic} representation spaces of $(\C^*)^n$.  Based on this observation we define our topological analogue of a toric manifold as follows. 

\begin{defi}
We say that a closed smooth manifold $X$ of dimension $2n$ with an effective \emph{smooth} action of $(\C^*)^n$ having an open dense orbit is a (compact) \emph{topological toric manifold} if it is covered by finitely many invariant open subsets each equivariantly diffeomorphic to a direct sum of complex one-dimensional \emph{smooth} representation spaces of $(\C^*)^n$.  
\end{defi}

We make two remarks.  The first one is that the existence of the open dense orbit of $(\C^*)^n$ does not imply the existence of the local charts diffeomorphic to smooth representation spaces of $(\C^*)^n$ (see the example in Section~\ref{sect:3}) although it does in the toric case.  The second one is that there are many more smooth representations of $(\C^*)^n$ than algebraic ones.  This stems from the fact that since $\C^*=\R_{>0}\times S^1$ as smooth groups, any \emph{smooth} endomorphism of $\C^*$ is of the form 
\begin{equation} \label{intro1}
\text{$g\mapsto |g|^{b+\sqrt{-1}c}\big(\frac{g}{|g|}\big)^v$\quad with $(b+\sqrt{-1}c,v)\in \RZR$}
\end{equation}
and this endomorphism is algebraic if and only if $b=v$ and $c=0$.  Therefore the group $\Hom(\C^*,\C^*)$ of \emph{smooth} endomorphisms of $\C^*$ is isomorphic to $\RZR$ while the group $\Hom_{alg}(\C^*,\C^*)$ of \emph{algebraic} endomorphisms of $\C^*$ is isomorphic to $\Z$.  This implies that topological toric manifolds are much more abundant than toric manifolds.  

Nevertheless, topological toric manifolds have similar topological properties to toric manifolds.  For instance, the orbit space of a topological toric manifold $X$ by the restricted compact torus action is a  manifold with corners whose faces (even the orbit space itself) are all contractible and any intersection of faces is connected unless it is empty, so the orbit space looks like a simple polytope.  This implies that the cohomology ring $H^*(X;\Z)$ of $X$ is generated by degree two elements as a ring like the toric or quasitoric case (Proposition~\ref{prop:cohom}).  However, the orbit space is not necessarily a simple polytope and one can see the following. 

\begin{theo}[see Theorem~\ref{theo:quasi} for more details]
The family of topological toric manifolds with restricted compact torus actions is strictly larger than the family of quasitoric manifolds up to equivariant homeomorphism.  
\end{theo}

As a combinatorial counterpart to a topological toric manifold, we introduce the notion of a \emph{simplicial topological fan} (we simply say \emph{topological fan} in this paper) generalizing the notion of a simplicial fan in toric geometry.  A simplicial fan of dimension $n$ is a collection of simplicial cones in $\R^n$ satisfying certain conditions.  It can be regarded as a pair $(\Si,v)$ of an abstract simplicial complex $\Si$ and a map $v\colon \Si^{(1)}\to \Z^n$, where $\Si$ is the underlying simplicial complex of the fan, $\Si^{(1)}$ is the set of vertices in $\Si$ which correspond to one-dimensional cones in the fan, and $v$ assigns primitive integral vectors lying on the one-dimensional cones.  We note that the target group $\Z^n$ of the map $v$ should actually be regarded as the group $\Hom_{alg}(\C^*,(\C^*)^n)$ of \emph{algebraic} homomorphisms from $\C^*$ to $(\C^*)^n$.   

We define a topological fan of dimension $n$ to be a pair $\De=(\Si,\beta)$ of an abstract simplicial complex $\Si$ and a map $\beta\colon \Si^{(1)}\to \Hom(\C^*,(\C^*)^n)$ satisfying certain conditions, where $\Hom(\C^*,(\C^*)^n)$ denotes the group of \emph{smooth} homomorphisms from $\C^*$ to $(\C^*)^n$.  We may think of a topological fan as a collection of cones in $\Hom(\C^*,(\C^*)^n)$ by forming cones $\mathcal C$ using the $\Si$ and $\beta$.  As observed in \eqref{intro1}, $\Hom(\C^*,(\C^*)^n)$ is isomorphic to $\RZRn$, so we may regard $\beta$ as a map to $\RZRn$ and write $\beta=(b+\sqrt{-1}c,v)$ accordingly.  Then an ordinary simplicial fan is a topological fan with $b=v$ and $c=0$.  Cones obtained from the pair $(\Si,b)$, which are the projected images of the cones $\mathcal C$ on the real part of the first factor of $\RZRn$, do not overlap and define an ordinary simplicial fan over $\R$ while cones formed from the pair $(\Si,v)$, which are the projected images of $\mathcal C$ on the second factor of $\RZRn$, may overlap and define a multi-fan.  A topological fan $\De$ is called \emph{complete} if the ordinary fan $(\Si,b)$ is complete, and \emph{non-singular} if the multi-fan $(\Si,v)$ is non-singular, i.e. if $\{v(\{i\})\}_{i\in I}$ is a part of a $\Z$-basis of $\Z^n$ for any $I\in \Si$.  

One can associate a complete non-singular topological fan $\De(X)$ of dimension $n$ to an omnioriented topological toric manifold $X$ of dimension $2n$.  Conversely, one can construct an omnioriented topological toric manifold $X(\De)$ from a complete non-singular topological fan $\De$ using the quotient construction of toric manifolds developed by Cox \cite{cox95}. Our main theorem in this paper is the following. 

\begin{theo}[Theorem~\ref{theo:class}]
The correspondences $X\to \De(X)$ and $\De\to X(\De)$ give bijections 
\[
\begin{split}
&\{\text{Omnioriented topological toric manifolds of dimension $2n$}\}\\
&\rightleftarrows\ \{\text{Complete non-singular topological fans of dimension $n$}\}
\end{split}
\]
and they are inverses to each other. 
\end{theo}

We also classify omnioriented topological toric manifolds up to equivariant diffeomorphism or equivariant homeomorphism in terms of the associated topological fans (Corollary~\ref{coro:class2}).   

This paper is organized as follows.  In Section~\ref{sect:2} we study smooth representations of $(\C^*)^n$.  In Section~\ref{sect:3} we deduce some properties of topological toric manifolds.  This tells us how to define topological fans.  In Section~\ref{sect:4} we construct a topological space $X(\De)$ with an action of $(\C^*)^n$ from a complete non-singular topological fan $\De=(\Si,\beta)$. We prove that $X(\De)$ with the action of $(\C^*)^n$ is actually a topological toric manifold in Sections~\ref{sect:5} and~\ref{sect:6}.  In Section~\ref{sect:7} we study $X(\De)$ as an $(S^1)^n$-manifold.  In particular, we prove that the $(S^1)^n$-equivariant homeomorphism type of $X(\De)$ does not depend on the first component of the $\beta$ (Theorem~\ref{theo:homeo}).  Our main theorem is proved in Section~\ref{sect:8}.  A topological toric manifold with the restricted action of the compact torus is a torus manifold.  So one can apply results on torus manifolds to topological toric manifolds.   Using a result from \cite{ma-pa06}, we describe the cohomology ring of a topological toric manifold $X$ in terms of the associated topological fan $\De(X)$ (Proposition~\ref{prop:cohom}).  We also describe the total Pontrjagin class of $X$. In Section~\ref{sect:9} we show that the Barnette sphere, which is a non-polytopal simplicial 3-sphere, can be the underlying simplicial complex of a topological toric manifold but cannot be that of a toric manifold.  We discuss the relation of topological toric manifolds with quasitoric manifolds in Section~\ref{sect:10} and with torus manifolds in Section~\ref{sect:11}.  In Section~\ref{sect:12}, we introduce \emph{real} topological toric manifolds and discuss their relation with small covers introduced by Davis-Januszkiewicz \cite{da-ja91} as a real version of quasitoric manifolds.  Finally, we collect a few facts on the underlying simplicial complexes of (topological) toric manifolds in the Appendix.

\section{Smooth representations of $(\C^*)^n$} \label{sect:2}

In this section we study smooth representations of $(\C^*)^n$ and set up some notations needed later.  We begin with the case where $n=1$. As is easily checked, any smooth group endomorphism of $\C^*$ is of the form 
\[
\text{$g\to |g|^{b+\sqrt{-1}c}\big(\frac{g}{|g|}\big)^v=:g^\mu$\quad with $\mu=(b+\sqrt{-1}c,v)\in \RZR$.}
\]
Since $\GL(1,\C)=\C^*$, a smooth group endomorphism of $\C^*$ can be regarded as a complex one-dimensional smooth representation of $\C^*$.  This representation is \emph{algebraic} if and only if $b=v$ and $c=0$.  We set $\bar\mu:=(b-\sqrt{-1}c,-v)$.  Then $\overline{g^{\mu}}=g^{\bar \mu}$. 

\begin{lemm} \label{lemm:gmu}
\begin{enumerate}
\item If $b>0$, then the endomorphism $g\to g^{\mu}$ extends continuously at $g=0$.  If $v=\pm 1$ in addition to $b>0$, then the extended map is a homeomorphism.   
\item $\bar g^\mu=\overline{g^{\mu}}$ for any $g\in \C^*$ if and only if $c=0$.  
\item $g^{\mu}=g^p\bar g^q$ for some integers $p,q$ if and only if $b$ is an integer congruent to $v$ modulo $2$ and $c=0$.  In fact, $g^{\mu}=g^{(b+v)/2}\bar g^{(b-v)/2}$ in this case.  
\item A complex one-dimensional smooth representation $g\to g^{\mu'}$ is isomorphic to the representation $g\to g^{\mu}$ as real representations if and only if $\mu'=\mu$ or $\bar\mu$.  
\end{enumerate}
\end{lemm}

\begin{proof}
Statements (1), (2) and (3) are clear.  As for (4), if $\mu'=\bar\mu$, then the complex conjugation map gives an isomorphism between those two representations, proving the \lq\lq if" part.  Suppose that the two representations in (4) are isomorphic as real representations. 
Let us denote $\mu = (b + \sqrt{-1}c,v)$ and $\mu '=(b'+\sqrt{-1}c',v')$. By the definition of $g^{\mu}$, the matrix presentation of $g^{\mu}$ with respect to the basis $(1,\sqrt{-1})$ of $\C$ as an $\R$-vector space is
\begin{equation*}\exp (bx)
	\begin{pmatrix}
		\cos (cx+vy)  &-\sin (cx+vy)\\
		\sin (cx+vy) &\ \ \ \cos (cx+vy)
	\end{pmatrix}
\end{equation*}
where $g=\exp(x+\sqrt{-1}y), x, y \in \R$. Hence the characters of the representations $g \to g^{\mu}$ and $g \to g^{\mu '}$ are respectively given as 
\begin{equation*}
	2\exp (bx)\cos (cx+vy) \quad \text{ and }\quad 2\exp (b'x)\cos (c'x+v'y).
\end{equation*}
These functions are same if and only if $b=b'$ and $(c,v)=\pm (c',v')$. 
This implies $\mu'=\mu$ or $\bar\mu$, proving the \lq\lq only if" part. 
\end{proof}

For $\mu_i=(b_i+\sqrt{-1}c_i,v_i)$ $(i=1,2)$ we have 
\[
\begin{split}
(g^{\mu_1})^{\mu_2}&=\big(|g|^{b_1+\sqrt{-1}c_1}(g/|g|)^{v_1}\big)^{\mu_2}\\
&=(|g|^{b_1})^{b_2+\sqrt{-1}c_2}\big(|g|^{\sqrt{-1}c_1}(g/|g|)^{v_1}\big)^{v_2}\\
&=|g|^{b_1b_2+\sqrt{-1}(b_1c_2+c_1v_2)}(g/|g|)^{v_1v_2},
\end{split}
\]
so if we define
\begin{equation} \label{prod}
\mu_2\mu_1:=(b_1b_2+\sqrt{-1}(b_1c_2+c_1v_2),v_1v_2),
\end{equation}
then we have 
\begin{equation} \label{expo}
(g^{\mu_1})^{\mu_2}=g^{\mu_2\mu_1}.
\end{equation}

It is better to regard $\mu=(b+\sqrt{-1}c,v)$ as a $2\times 2$ matrix 
\begin{equation} \label{matr}
\begin{bmatrix}
b &0\\
c&v
\end{bmatrix}.
\end{equation}
Then \eqref{prod} is nothing but the identity   
$$
\begin{bmatrix}
b_2 &0\\
c_2&v_2
\end{bmatrix} \begin{bmatrix}
b_1 &0\\
c_1&v_1
\end{bmatrix}=\begin{bmatrix}
b_2b_1 & 0\\
c_2b_1+v_2c_1 & v_2v_1
\end{bmatrix}.
$$

Let $\MR$ be a ring consisting of $2\times 2$ matrices of the form \eqref{matr} with $b,c\in \R$ and $v\in\Z$.  The identity matrix $\1$ (resp. zero matrix $\0$) of size 2 is the identity (resp. zero) element of the ring $\MR$. The set $\Hom(\C^*,\C^*)$ of smooth endomorphisms of $\C^*$ forms a ring where the addition of $\rho_1,\rho_2\in \Hom(\C^*,\C^*)$ is the endomorphism $g\to \rho_1(g)\rho_2(g)$ and the multiplication of them is their composition.  The observation above shows that $\Hom(\C^*,\C^*)$ is isomorphic to $\MR$ as rings. 

For $\alpha=(\alpha^1,\dots,\alpha^n)\in \MR^n$ and $\beta=(\beta^1,\dots,\beta^n)\in \MR^n$, we define smooth homomorphisms $\ch^\alpha\in\Hom((\C^*)^n,\C^*)$ and $\la_\beta\in \Hom(\C^*,(\C^*)^n)$ by 
\begin{equation} \label{chla}
\ch^\alpha(g_1,\dots,g_n):=\prod_{k=1}^n g_k^{\alpha^k},  \quad\la_\beta(g):=(g^{\beta^1},\dots,g^{\beta^n}).
\end{equation}
We also define
\begin{equation} \label{inner}
\langle \alpha,\beta\rangle:=\sum_{k=1}^n \alpha^k\beta^k\in \MR.  
\end{equation}

\begin{lemm} \label{lemm:chlalach}
\begin{enumerate}
\item $(g_1,\dots,g_n)\in (\C^*)^n$ is the identity element if and only if $\ch^{\alpha}(g_1,\dots,g_n)=1$ for any $\alpha\in \MR^n$. 
\item $\ch^\alpha(\la_{\beta}(g))=g^{\langle \alpha,\beta\rangle}$.
\item $\displaystyle{\la_{\beta}(\ch^{\alpha}(g_1,\dots,g_n))=(\prod_{k=1}^ng_k^{\beta^1\alpha^k},\dots, \prod_{k=1}^ng_k^{\beta^n\alpha^k})}$.
\end{enumerate}
\end{lemm}

\begin{proof}
(1) The \lq\lq only if" part is trivial.  If $\alpha$ has $\1$ at the $i$-th entry and $\0$ at the other entries, then $\ch^{\alpha}(g_1,\dots,g_n)=g_i$.  This implies the \lq\lq if" part. 

Statements (2) and (3) easily follow from \eqref{expo}, \eqref{chla} and \eqref{inner}. 
\end{proof}

For $\{\alpha_i\}_{i=1}^n$ and $\{\beta_i\}_{i=1}^n$ with $\alpha_i,\beta_i\in\MR^n$, we define group endomorphisms $\bigoplus_{i=1}^n\ch^{\alpha_i}$ and $\prod_{i=1}^n\la_{\beta_i}$ of $(\C^*)^n$ by 
\begin{equation} \label{defprod}
\begin{split}
&(\bigoplus_{i=1}^n\ch^{\alpha_i})(g_1,\dots,g_n):=(\ch^{\alpha_1}(g_1,\dots,g_n),\dots, \ch^{\alpha_n}(g_1,\dots,g_n))\\
&(\prod_{i=1}^n\la_{\beta_i})(g_1,\dots,g_n):=\prod_{i=1}^n\la_{\beta_i}(g_i).
\end{split}
\end{equation}

With this understood, we have 

\begin{lemm} \label{lemm:comp}
If $\{\alpha_i\}_{i=1}^n$ is dual to $\{\beta_i\}_{i=1}^n$, i.e. $\langle \alpha_i,\beta_j\rangle=\delta_{ij}\1$ where $\delta_{ij}$ denotes the Kronecker delta, then the composition 
$$\big(\prod_{i=1}^n \la_{\beta_i}\big)\big(\bigoplus_{i=1}^n\ch^{\alpha_i}\big)\colon (\C^*)^n\to (\C^*)^n$$ 
is the identity,  in particular, both $\bigoplus_{i=1}^n \ch^{\alpha_i}$ and $\prod_{i=1}^n \la_{\beta_i}$ are automorphisms of $(\C^*)^n$.  
\end{lemm}

\begin{proof}
It follows from \eqref{defprod} and Lemma~\ref{lemm:chlalach} (2) that if we write $\alpha_i=(\alpha_i^1,\dots,\alpha_i^n)$ and $\beta_i=(\beta_i^1,\dots,\beta_i^n)$, then   
\begin{equation}  \label{keisan}
\begin{split}
&\big(\prod_{i=1}^n \la_{\beta_i}\big)\big(\bigoplus_{i=1}^n\ch^{\alpha_i}\big)(g_1,\dots,g_n)\\
=&(\prod_{i=1}^n\la_{\beta_i})\big(\chi^{\alpha_1}(g_1,\dots,g_n),\dots,\chi^{\alpha_n}(g_1,\dots,g_n)\big)\\
=& \prod_{i=1}^n\la_{\beta_i}\big(\chi^{\alpha_i}(g_1,\dots,g_n)\big)
=\prod_{i=1}^n\big(\prod_{k=1}^ng_k^{\beta_i^1\alpha_i^k},\dots,\prod_{k=1}^ng_k^{\beta_i^n\alpha_i^k}\big)\\
=&\big(\prod_{k=1}^ng_k^{\sum_{i=1}^n\beta_i^1\alpha_i^k},\dots,\prod_{k=1}^ng_k^{\sum_{i=1}^n\beta_i^n\alpha_i^k}\big).
\end{split}
\end{equation}
On the other hand, the assumption that $\langle \alpha_i,\beta_j\rangle=\delta_{ij}\1$ is equivalent to the following identity
\begin{equation} \label{bigmatrix}
\begin{bmatrix} \alpha^1_1 & \alpha^2_1 & \dots &\alpha^n_1 \\
\alpha^1_2 & \alpha^2_2 & \dots &\alpha^n_2 \\
\vdots &\vdots & \ddots & \vdots\\
\alpha^1_n & \alpha^2_n & \dots &\alpha^n_n 
\end{bmatrix} 
\begin{bmatrix} \beta^1_1& \beta^1_2 & \dots &\beta^1_n\\
\beta^2_1& \beta^2_2 & \dots &\beta^2_n\\
\vdots &\vdots & \ddots & \vdots\\
\beta^n_1& \beta^n_2 & \dots &\beta^n_n
\end{bmatrix} =\begin{bmatrix} \1 & \0 &\dots &\0\\
\0&\1&\dots &\0\\
\vdots &\vdots & \ddots & \vdots\\
\0&\0&\dots &\1\end{bmatrix}
\end{equation}
Each entry in \eqref{bigmatrix} is a matrix of size $2$, so that the product at the left hand side of \eqref{bigmatrix} can be viewed as a product of matrices of size $2n$ with entries in $\R$ or $\Z$.  Since the product is the identity matrix,  we may interchange the two matrices at the left hand side of \eqref{bigmatrix}.  Then the resulting identity shows that $\sum_{i=1}^n\beta_i^j\alpha_i^k=\delta_{jk}\1$ and this means that the last term in \eqref{keisan} reduces to $(g_1,\dots,g_n)$, proving the lemma.    
\end{proof}

Through the identification of $\MR$ with $\RZR$, we identify $\MR^n$ with $\RZRn$ and write $\beta_i\in \MR^n$ as $(b_i+\sqrt{-1}c_i,v_i)\in \RZRn$.  With this understood, we have

\begin{lemm} \label{lemm:dual}
If $\{b_i\}_{i=1}^n$ and $\{v_i\}_{i=1}^n$ are bases of $\R^n$ and $\Z^n$ respectively, then $\{\beta_i\}_{i=1}^n$ has a dual set $\{\alpha_i\}_{i=1}^n$, i.e. $\langle \alpha_i,\beta_j\rangle=\delta_{ij}\1$ for any $i,j$.   
\end{lemm}  

\begin{proof}
Write $\beta_i=(\beta_i^1,\dots,\beta_i^n)$ as before and $\beta_i^k=(b_i^k+\sqrt{-1}c_i^k,v_i^k)$. Setting  
\[
B=\begin{bmatrix} b^1_1& b^1_2 & \dots &b^1_n\\
b^2_1& b^2_2 & \dots &b^2_n\\
\vdots &\vdots & \ddots & \vdots\\
b^n_1& b^n_2 & \dots &b^n_n
\end{bmatrix}, \quad 
C=\begin{bmatrix} c^1_1& c^1_2 & \dots &c^1_n\\
c^2_1& c^2_2 & \dots &c^2_n\\
\vdots &\vdots & \ddots & \vdots\\
c^n_1& c^n_2 & \dots &c^n_n
\end{bmatrix}, \quad
V=\begin{bmatrix} v^1_1& v^1_2 & \dots &v^1_n\\
v^2_1& v^2_2 & \dots &v^2_n\\
\vdots &\vdots & \ddots & \vdots\\
v^n_1& v^n_2 & \dots &v^n_n
\end{bmatrix},  
\]
we see that the second matrix at the left hand side of \eqref{bigmatrix} is conjugate to this block matrix 
$
\begin{bmatrix}
B&0\\
C&V
\end{bmatrix}
$
by a permutation matrix $P$.  The assumption in the lemma implies that $B\in \GL(n,\R)$ and $V\in \GL(n,\Z)$.  Therefore the above block matrix has an inverse  
$
\begin{bmatrix}
B^{-1}&0\\
-V^{-1}CB^{-1}&V^{-1}
\end{bmatrix}.
$
Conjugating this matrix by $P^{-1}$, we obtain the first matrix at the left hand side in \eqref{bigmatrix}, proving the existence of the dual set $\{\alpha_i\}_{i=1}^n$. 
\end{proof}

Consider a complex $n$-dimensional smooth representation of $(\C^*)^n$ of the form $\bigoplus_{i=1}^n\ch^{\alpha_i}$ with $\alpha_i\in\MR^n$. This representation is algebraic if and only if every factor $\ch^{\alpha_i}$ is algebraic, so complex $n$-dimensional \emph{algebraic} representations of $(\C^*)^n$ are parameterized by $(\Z^n)^n$ while complex $n$-dimensional \emph{smooth} ones are parameterized by $(\MR^n)^n=(\RZRn)^n$. 

Throughout this paper we will denote the representation space of a representation $\rho$ by $V(\rho)$. 
\begin{lemm} \label{lemm:faithful}
\begin{enumerate}
\item A smooth complex $n$-dimensional representation $\bigoplus_{i=1}^n\ch^{\alpha_i}$ of $(\C^*)^n$ is faithful if and only if the matrix in \eqref{bigmatrix} formed from $\alpha_i$'s has an inverse with entries in $\MR$.    
\item If $\bigoplus_{i=1}^n\ch^{\alpha_i}$ is faithful, then the representation space $\bigoplus_{i=1}^nV(\ch^{\alpha_i})$ has an open dense orbit and has only finitely many orbits.  In fact, the open dense orbit is unique and isomorphic to $(\C^*)^n$ and the other orbits have real codimension at least two. 
\end{enumerate}
\end{lemm}
\begin{proof}
	The statement (1) would be clear. To see (2), we suppose that $\bigoplus_{i=1}^n\chi ^{\alpha _i}$ is faithful. Then, the homomorphism $\bigoplus_{i=1}^n\chi ^{\alpha _i} : (\C ^*)^n \to (\C ^*)^n$ has the inverse $\prod _{i=1}^n\lambda _{\beta _i}$, where $\{ \beta _i\}$ is dual to $\{ \alpha _i\}$. Let $z_i$ be a coordinate function of $V(\chi ^{\alpha _i})$. For a point $p \in \bigoplus _{i=1}^nV(\chi ^{\alpha _i})$, we set
	\begin{equation*}
		I(p):= \{ i \mid z_i(p) = 0\}.
	\end{equation*}
	We will see that two points $p$ and $q$ in $\bigoplus _{i=1}^nV(\chi ^{\alpha _i})$ belong to the same orbit if and only if $I(p)=I(q)$. The \lq\lq{only if}\rq\rq part is obvious. So we assume that $I(p)=I(q)$. Then, the fractions $z_i(p)/z_i(q)$ for $i \notin I(p)$ is an element in $\C ^*$ and $z_i(p)=z_i(q)=0$ if $i \in I(p)$. We take an element $g=(g_1,\dots ,g_n) \in (\C ^*)^n$ to be
	\begin{equation*}
		g_i= \begin{cases}
			z_i(p)/z_i(q) & \text{if $i \notin I(p)$},\\
			1 & \text{if $i \in I(p)$}. 
		\end{cases}
	\end{equation*}
	Let $h = \prod _{i=1}^n\lambda _{\beta _i}(g_i)$. Then, it follows from (2) in Lemma \ref{lemm:chlalach} that
	\begin{equation*}
		z_i(h\cdot q)=\chi ^{\alpha _i}(h)z_i(q)=\chi ^{\alpha _i}\big( \prod_{j=1}^n\lambda _{\beta _j}(g_j )\big)z_i(q)=g_iz_i(q)=z_i(p)
	\end{equation*} 
	for all $i =1,\dots ,n$. Namely, $h\cdot q=p$ and hence $p$ and $q$ belong to the same orbit. Hence the \lq\lq{if}\rq\rq part holds. Therefore, the representation space $\bigoplus_{i=1}^nV(\ch^{\alpha_i})$ has an open dense orbit which consists of points $p$ with $I(p)=\emptyset$. Other orbits correspond to nonempty subsets $I$ of \{1,\dots ,n\} and have real codimension $2|I|$. The lemma is proved. 
\end{proof}
Suppose that $\bigoplus_{i=1}^n\ch^{\alpha_i}$ is faithful.  Then the representation space $\bigoplus_{i=1}^nV(\ch^{\alpha_i})$ has an open dense orbit isomorphic to $(\C^*)^n$ as mentioned above.  Since $(\C^*)^n$ has a canonical orientation, $\bigoplus_{i=1}^nV(\ch^{\alpha_i})$ has an orientation induced from the open dense orbit.  On the other hand, $\bigoplus_{i=1}^nV(\ch^{\alpha_i})$ has an orientation induced as the complex vector space.  These two orientations agree when the representation is algebraic.  But otherwise, they may disagree.  For instance, if $n=1$ and our representation of $\C^*$ is given by the complex conjugation $g\to \bar g$, then those two orientations disagree.  The following lemma would be easy to see. 

\begin{lemm} \label{lemm:orient}
Suppose that $\bigoplus_{i=1}^n\ch^{\alpha_i}$ is faithful and let $\{\beta_i\}_{i=1}^n$ be dual to $\{\alpha_i\}_{i=1}^n$.  Then the two orientations on $\bigoplus_{i=1}^nV(\ch^{\alpha_i})$ mentioned above agree if and only if the determinant of the matrix in \eqref{bigmatrix} formed from the $\alpha_i$'s 
or $\beta_i$'s (viewed as a matrix of size $2n$) is positive.  
\end{lemm}

We conclude this section with the classification of faithful complex $n$-dimensional representation spaces of $(\C^*)^n$ up to equivariant diffeomorphism or homeomorphism.  

\begin{lemm} \label{lemm:repclass}
Let $\bigoplus_{i=1}^nV(\ch^{\alpha_i})$ and $\bigoplus_{i=1}^nV(\ch^{\alpha_i'})$ be complex $n$-dimensional faithful representation spaces of $(\C^*)^n$.   Let $\beta _i$'s (resp. $\beta_i'$'s) be dual to $\alpha _i$'s (resp. $\alpha_i'$'s).  
\begin{enumerate}
\item The above representation spaces are equivariantly diffeomorphic if and only if there is a permutation $\sigma$ on $[n]=\{1,2,\dots,n\}$ such that $\beta'_{\sigma(i)}=\beta_i$ or $\beta_i\mu_0$ for any $i$ where $\mu_0=\begin{bmatrix} 1 &0\\
0 & -1\end{bmatrix}$.
\item The above representation spaces are equivariantly homeomorphic if and only if there is a permutation $\sigma$ on $[n]$ and some $\mu _i$ in 
\begin{equation*}
\MS := \left\{
\begin{bmatrix}
	b & 0 \\
	c & v 
\end{bmatrix} \in \MR \mid b>0, v=\pm 1 
\right\} 
\end{equation*}
such that $\beta'_{\sigma(i)}=\beta_i\mu_i$ for any $i$. 
\end{enumerate}
\end{lemm}

\begin{proof}
The \lq\lq if" part in (1) is clear.  Suppose that there is an equivariant diffeomorphism $f\colon \bigoplus_{i=1}^nV(\ch^{\alpha_i})\to \bigoplus_{i=1}^nV(\ch^{\alpha_i'})$.  Since the origin is the unique fixed point in these representation spaces and their tangential representation space at the origins are themselves, the differential $df_0$ of $f$ at the origin induces an isomorphism between $\bigoplus_{i=1}^n\ch^{\alpha_i}$ and $\bigoplus_{i=1}^n\ch^{\alpha_i'}$ as \emph{real} representations. Since the representation $\bigoplus_{i=1}^n\ch^{\alpha_i}$ (resp. $\bigoplus_{i=1}^n\ch^{\alpha_i'}$) is faithful, the factors $\ch^{\alpha_i}$ (resp. $\ch^{\alpha_i'}$) are mutually non-isomorphic as real representations.  Therefore, $df_0$ maps $V(\ch^{\alpha_i})$ to some factor of $\bigoplus_{i=1}V(\ch^{\alpha_i'})$ isomorphically as real representation spaces, so there is a permutation $\sigma$ on $[n]$ such that $\alpha_i=\alpha'_{\sigma(i)}$ or $\bar{\alpha}'_{\sigma(i)}$ by Lemma~\ref{lemm:gmu} (4). This implies the \lq\lq only if" part in (1) because $\bar{\alpha}'_{\sigma(i)}=\mu_0\alpha'_{\sigma(i)}$ and the $\beta _i$'s (resp. $\beta_i'$'s) are dual to the $\alpha _i$'s (resp. $\alpha_i'$'s).

The proof of  statement (2) is as follows.  We shall prove the \lq\lq if" part first.  We may assume that $\sigma$ is the identity without loss of generality.  Since $\mu_i$ is in $\mathcal{S}$, so is $\mu_i^{-1}$ and the map 
\[
(z_1,\dots,z_n) \to (z_1^{\mu_1^{-1}},\dots,z_n^{\mu_n^{-1}})
\]
is a homeomorphism of $\C^n$ by Lemma~\ref{lemm:gmu} (1) and it easily follows from \eqref{expo} that it gives an equivariant homeomorphism from  $\bigoplus_{i=1}^nV(\ch^{\alpha_i})$ to $\bigoplus_{i=1}^nV(\ch^{\alpha_i'})$ with $\alpha_i'=\mu_i^{-1}\alpha_i$ for any $i$.  This implies the \lq\lq if" part because the $\beta _i$'s (resp. $\beta_i'$'s) are dual to the $\alpha _i$'s (resp. $\alpha_i'$'s).  

Conversely, suppose there is an equivariant homeomorphism $f\colon \bigoplus_{i=1}^nV(\ch^{\alpha_i})\to \bigoplus_{i=1}^nV(\ch^{\alpha_i'})$.  The fixed point set in $\bigoplus_{i=1}^nV(\ch^{\alpha_i})$ under the action of the $\C^*$-subgroup $\lambda _{\beta _i}(\C ^*)$ is $\bigoplus_{j\not=i}V(\chi^{\alpha_j})=:V_i$ which is of complex codimension one.  Note that any $\C^*$-subgroup of $(\C^*)^n$ whose fixed point set in $\bigoplus_{i=1}^nV(\ch^{\alpha_i})$ is of complex dimension one is $\lambda _{\beta _i}(\C ^*)$ for some $i$.  Since the same is true for $\bigoplus_{i=1}^nV(\ch^{\alpha_i'})$ and the map $f$ above is equivariant homeomorphism, one concludes that there is a permutation $\sigma$ on $[n]$ such that 
\[
f(V_{i})=V'_{\sigma(i)}\quad\text{and}\quad \lambda_{\beta_i}(\C ^*)=\lambda_{\beta'_{\sigma(i)}}(\C ^*)\quad\text{for each $i$},
\]
where $V_i'$ is defined similarly to $V_i$.  The latter identity above implies that there is an element $\mu_i=
\begin{bmatrix}
	b_i & 0 \\
	c_i & v_i 
\end{bmatrix}\in \MR$ such that 
\begin{equation*} \label{bb}
\beta'_{\sigma(i)}=\beta_i\mu_i \quad\text{for each $i$}.
\end{equation*}
Here $v_i=\pm 1$ because $\lambda_{\beta'_{\sigma(i)}}$ and $\lambda_{\beta_i}$ are both injective.  Moreover, $b_i>0$ because for any generic point $\iota $ in $\bigoplus_{i=1}^nV(\ch^{\alpha_i})$, $\lim_{g \to 0}\lambda_{\beta_i}(g)(\iota)$ (resp. $\lim_{g \to 0}\lambda_{\beta'_{\sigma(i)}}(g)(f(\iota))$) converges to a point in $V_i$ (resp. $V'_{\sigma(i)}$) and $f(V_i)=V'_{\sigma(i)}$. Hence, $\mu _i$ lies in $\MS$, proving the \lq\lq{only if}\rq\rq part in statement (2). 
\end{proof}

\section{Topological toric manifold and topological fan} \label{sect:3}

A \emph{toric variety} of complex dimension $n$ is a normal algebraic variety of complex dimension $n$ with an effective \emph{algebraic} action of $(\C^*)^n$ having an open dense orbit.  A toric variety has only finitely many orbits.  A compact smooth toric variety is called a \emph{toric manifold}.  It is known that a toric manifold is covered by finitely many invariant open subsets each equivariantly and algebraically isomorphic to a direct sum of complex one-dimensional \emph{algebraic} representation spaces of $(\C^*)^n$.  Based on this observation we define our topological analogue of a toric manifold as follows. 

\begin{defi}
We say that a closed smooth manifold $X$ of dimension $2n$ with an effective \emph{smooth} action of $(\C^*)^n$ having an open dense orbit is a  \emph{topological toric manifold} if it is covered by finitely many invariant open subsets each equivariantly diffeomorphic to a direct sum of complex one-dimensional \emph{smooth} representation spaces of $(\C^*)^n$.  
\end{defi}

\begin{rema}
Since the action of $(\C^*)^n$ on the topological toric manifold $X$ is assumed to be effective, the direct sum of complex one-dimensional \emph{smooth} representation spaces of $(\C^*)^n$ in the definition above must be faithful.  On the other hand, it is not difficult to see that any faithful smooth real $2n$-dimensional representation of $(\C^*)^n$ is isomorphic to a direct sum of complex one-dimensional representations as real representations. Therefore, we may replace \lq\lq a direct sum of complex one-dimensional smooth representation spaces of $(\C^*)^n$" in the definition above by \lq\lq a smooth representation space of $(\C^*)^n$".
\end{rema}

\begin{lemm} \label{lemm:open}
A topological toric manifold has only finitely many orbits and the real codimension of an orbit different from the open dense orbit is at least two. 
\end{lemm}

\begin{proof}
The effectiveness of the action of $(\C^*)^n$ in $X$ implies the faithfulness of the complex $n$-dimensional representations in the definition of a topological toric manifold above.  Therefore the lemma follows from Lemma~\ref{lemm:faithful} (2). 
\end{proof}

In particular, the fixed point set $X^{(\C^*)^n}$ of a topological toric manifold $X$ consists of finitely many points. It follows from the definition of a topological toric manifold that each fixed point $p\in X^{(\C^*)^n}$ has an invariant open neighborhood, denoted $O_p$, which is equivariantly diffeomorphic to a faithful representation space of $(\C^*)^n$.  Then we have
$$X=\bigcup_{p\in X^{(\C^*)^n}}O_p.$$

\begin{prop} \label{prop:simply}
A topological toric manifold is simply connected, in particular, orientable. 
\end{prop}

\begin{proof}
The open dense orbit in a topological toric manifold $X$, denoted $O$, can be identified with $(\C^*)^n$.  Since it is connected, its closure, that is $X$, is also connected.  

Any orbit in $X$ except $O$ has codimension at least two and the number of the orbits is finite by Lemma~\ref{lemm:open}.  Therefore the homomorphism $\kappa_* \colon \pi_1(O)\to \pi_1(X)$ induced from the inclusion map $\kappa\colon O\to X$ is surjective.  On the other hand, since $\kappa$ factors through the inclusion maps $O\to O_p\to X$ and $O_p$ is simply connected because it is diffeomorphic to a representation space, $\kappa_*$ must be trivial.  This proves that $\pi_1(X)$ is trivial.  
\end{proof}

In the algebraic category, i.e. in the case of toric manifolds, the existence of an open dense orbit implies the existence of a fixed point.  However, this is not always the case in the smooth category as is shown in the following example.  

\begin{exam}
Take a smooth vector field on $S^1$ having only one zero point, say $x$.  Integrating the vector field produces a smooth action of $\R$ on $S^1$ which has two orbits $\{x\}$ and $S^1\backslash\{x\}$.  We regard the action of $\R$ as an action of $\R_{>0}$ through a logarithmic function  $\R_{>0}\to \R$.  On the other hand, we take another $S^1$ with a natural free $S^1$-action.  Their product $S^1\times S^1$ supports a smooth action of $\R_{>0}\times S^1=\C^*$.  This action of $\C^*$ on $S^1\times S^1$ has no fixed point while it has two orbits, $\{x\}\times S^1$ and $(S^1\backslash\{x\})\times S^1$, and the latter orbit is open and dense in $S^1\times S^1$.  
\end{exam}

We say that a closed \emph{connected} smooth submanifold of a topological toric manifold $X$ of real codimension two is \emph{characteristic} if it is fixed pointwise under some $\C^*$-subgroup of $(\C^*)^n$.  There are only finitely many characteristic submanifolds, which we denote by $X_1,\dots,X_m$.  We define 
\[
\SiX:=\{ I\subset [m]\mid X_I:=\bigcap_{i\in I}X_i\not=\emptyset\}\cup \{\emptyset\}.
\]
This is an abstract simplicial complex of dimension $n-1$ and pure, i.e. any simplex in $\SiX$ is contained in some simplex of maximal dimension $n-1$.    We will denote by $\Si^{(k)}(X)$ the subset of $\Si(X)$ consisting of all $(k-1)$-simplices in $\Si(X)$.  Note that $\Si^{(1)}(X)$ can be identified with $[m]$. 

Since $X$ is locally equivariantly diffeomorphic to a direct sum of complex one-dimensional smooth representation spaces of $(\C^*)^n$, the $X_i$'s intersect transversally so that $X_I$ is a closed smooth submanifold of dimension $2(n-|I|)$ for $I\in \Si(X)$,  in particular, $X_I$ is of dimension $0$ when $I\in \Si^{(n)}(X)$ and   
\begin{equation} \label{fixed}
X^{(\C^*)^n}=\bigcup_{I\in \Si^{(n)}(X)}X_I.
\end{equation}
We will see in Lemma~\ref{lemm:XI} below that $X_I$ is connected for any $I\in \Si(X)$, in particular, $X_I$ is one point when $I\in \Si^{(n)}(X)$. 

Since $X_i$ is fixed pointwise under some $\C^*$-subgroup of $(\C^*)^n$, the normal bundle $\nu_i$ of $X_i$ to $X$ is orientable, so that each $X_i$ is orientable because so is $X$ by Proposition~\ref{prop:simply}.  

\medskip
\noindent
{\bf Convention.} 
Since a topological toric manifold $X$ of dimension $2n$ has an open dense orbit which can be identified with $(\C^*)^n$ and $(\C^*)^n$ has a canonical orientation, we give $X$ the orientation induced from the orientation of $(\C^*)^n$ throughout this paper unless otherwise stated.  

\medskip

A choice of an orientation on each $X_i$ together with the orientation on $X$ is called an \emph{omniorientation} on $X$.  An omniorientation on $X$ determines a compatible orientation on the normal bundle $\nu_i$.
Let $C_i$ be the $\C^*$-subgroup of $(\C^*)^n$ which fixes $X_i$ pointwise.  It acts on the normal bundle $\nu_i$ effectively through the differential and preserves each fiber.  As is easily checked, a real two-dimensional faithful representation space $V$ of $\C^*$ has exactly two complex structures invariant under the circle subgroup $S^1$ of $\C^*$ and they have different orientations.  The action of $g\in S^1$ on $V$ with the complex structures is scalar multiplication by $g$ or $g^{-1}$ according to the complex structures.  Moreover, the action of $\C^*$ on $V$ preserves the complex structure (but the action of $g\in\C^*$ on $V$ is not necessarily scalar multiplication by $g$ or $g^{-1}$, see Section~\ref{sect:2}).  Applying this fact to each fiber of $\nu_i$, one sees that $\nu_i$ admits a unique complex structure which is invariant under the circle subgroup of $C_i$ and whose induced orientation is compatible with the omniorientation.  Moreover, the action of $C_i$ on $\nu_i$ preserves the complex structure on $\nu_i$. With this understood, we have 
\begin{lemm}  \label{lemm:betai}
For each $i\in [m]=\Si^{(1)}(X)$, there is a unique $\beta_i(X)\in \MR^n$ such that 
\begin{equation} \label{norm}
\text{$\la_{\beta_i(X)}(\C^*)$ fixes $X_i$ pointwise and  $\la_{\beta_i(X)}(g)_*\xi=g\xi$ for any $g\in\C^*, \xi\in\nu_i$},
\end{equation}
where $\la_{\beta_i(X)}(g)_*$ denotes the differential of  $\la_{\beta_i(X)}(g)$ and the right hand side denotes scalar multiplication by $g$.  
\end{lemm}
\begin{proof}
Choose any group isomorphism $\gamma_i\colon \C^*\to C_i(\subset (\C^*)^n)$.  Then, for each fiber $\nu_i|_x$ of $\nu_i$ over $x\in X_i$, one has 
\[
\text{$\la_{\gamma_i}(g)_*\xi=\rho_i^x(g)\xi$\quad for any $g\in\C^*, \xi\in \nu_i|_x$}
\]
with some automorphism $\rho_i^x$ of $\C^*$.  However, $\rho_i^x$ is locally constant on $x$ because $X$ is covered by representation spaces of $(\C^*)^n$. Since $X_i$ is connected, it follows that $\rho_i^x$ is actually independent of $x\in X_i$; so we may denote $\rho_i^x$ by $\rho_i$.  Then, $\gamma_i\circ \rho_i^{-1}$ will be the desired $\beta_i(X)$.  The uniqueness of $\beta_i(X)$ would be obvious. 
\end{proof}

Note that if we reverse the orientation on $X_i$, then the orientation on $\nu_i$ will be reversed and hence the complex structure on $\nu_i$ will become complex conjugate to the original one.  

If a fixed point $p$ is in $X_I$ for $I\in\Si^{(n)}$, then 
\begin{equation} \label{TpX}
\tau_pX=\bigoplus_{i\in I}\nu_i|_p\quad\text{as real $(\C^*)^n$-representation spaces}
\end{equation}
where $\tau_pX$ denotes the tangential representation space of $X$ at $p$. 
Since $X$ is oriented by our convention, the left hand side of \eqref{TpX} has the orientation while the omniorientation on $X$ determines an orientation at the right hand side of \eqref{TpX}, and these two orientations may not agree.  For instance, if we reverse the orientation on only one $X_i$ for $i\in I$, then the orientation at the right hand side of \eqref{TpX} will be reversed while that at the left hand side of \eqref{TpX} will remain unchanged.  

We shall write $\beta_i(X)$ in Lemma~\ref{lemm:betai} as 
\[
\beta_i(X)=(b_i(X)+\sqrt{-1}c_i(X),v_i(X))\in \MR^n=\RZRn.
\]

\begin{lemm} \label{lemm:lindep}
Suppose that $X$ is omnioriented and let $I\in\Si^{(n)}(X)$.  Then $\{b_i(X)\}_{i\in I}$ and $\{v_i(X)\}_{i\in I}$ are bases of $\R^n$and $\Z^n$ respectively.  Therefore, $\{\beta_i(X)\}_{i\in I}$ has a dual set, denoted $\{\alpha^I_i(X)\}_{i\in I}$, and then the right hand side of \eqref{TpX} is isomorphic to $\bigoplus_{i\in I}V(\ch^{\alpha^I_i(X)})$ as complex representations.  In particular, the invariant open neighborhood $O_p$ of $p$ is equivariantly diffeomorphic to $\bigoplus_{i\in I}V(\ch^{\alpha^I_i(X)})$ and $O_p\cap X_i\not=\emptyset$ if and only if $i\in I$.  
\end{lemm}

\begin{proof} 
Since the action of $(\C^*)^n$ on $X$ is effective, the endomorphism $\prod_{i\in I}\la_{\beta_i(X)}$ of $(\C^*)^n$ is injective and hence an automorphism.  This implies  that $\{b_i(X)\}_{i\in I}$ and $\{v_i(X)\}_{i\in I}$ are bases of $\R^n$and $\Z^n$ respectively. It follows from \eqref{norm} and Lemma~\ref{lemm:chlalach} (2) that the the right hand side of \eqref{TpX} is isomorphic to $\bigoplus_{i\in I}V(\ch^{\alpha^I_i(X)})$ as complex representations. 
\end{proof}

For $I\in\SiX$ we denote by $\angle b_I(X)$ the cone spanned by $b_i(X)$'s for $i\in I$.  The dimension of $\angle b_I(X)$ is equal to the cardinality of $I$ by Lemma~\ref{lemm:lindep}.  

\begin{lemm} \label{lemm:limit}
Let $\iota$ be a point in the open dense orbit $O$ of $X$.  Let $I\in \Si^{(n)}(X)$ and $p\in X_I$.  Then the following hold. 
\begin{enumerate}
\item If $b\in\angle b_I(X)$, then one can write $b=\sum_{i\in I} r_ib_i(X)$ with $r_i\ge 0$.  In this case $\lim_{g\to 0}\la_{(b,0)}(g)(\iota)$ exists in $O_p\cap X_K$ where $K=\{ i\in I\mid r_i >0\}$.  In particular, if $b$ lies in the interior of $\angle b_I(X)$, then $\lim_{g\to 0}\la_{(b,0)}(g)(\iota)=p$. 
\item If $b\notin\angle b_I(X)$, then the sequence $\la_{(b,0)}(g_\ell)(\iota)$ does not converge in $O_p$ for any sequence $g_\ell$ $(\ell=1,2,\dots)$ in $\C^*$ approaching $0$.
\end{enumerate}
\end{lemm}

\begin{proof}
By Lemma~\ref{lemm:lindep} we may identify $O_p$ with $\bigoplus_{i\in I}V(\ch^{\alpha^I_i(X)})$ where $p$ corresponds to the origin in the representation space.   Since $\iota$ can be taken to be any point in the open dense orbit, we may assume $\iota=(1,\dots,1)$ through the identification.  Therefore, 
\begin{equation} \label{limit1}
\begin{split}
\la_{(b,0)}(g)(\iota)&=(\chi^{\alpha^I_i(X)}(\la_{(b,0)}(g)))_{i\in I}\\
&=(g^{\langle \alpha^I_i(X),(b,0)\rangle})_{i\in I} \qquad\text{in $O_p=\bigoplus_{i\in I}V(\ch^{\alpha^I_i(X)})$}
\end{split}
\end{equation}
and the absolute value of the $i$-component above is given by 
\begin{equation} \label{limit2}
\text{$|\ch^{\alpha^I_i(X)}(\la_{(b,0)}(g))|=|g^{\langle \alpha^I_i(X),(b,0)\rangle}|=|g|^{\langle a_i,b\rangle}$ for $g\in \C^*$}
\end{equation}
where $a_i\in \R^n$ denotes the real part of the first factor of $\alpha^I_i(X)\in \MR^n=\RZRn$. 

(1) If $b\in \angle b_I(X)$, then $\langle a_i,b\rangle=r_i\ge 0$ for any $i\in I$.  Hence $\langle a_i,b\rangle> 0$ for $i\in K$ and $\langle a_i,b\rangle=0$ for $i\in I\backslash K$.  This together with \eqref{limit1} and \eqref{limit2} implies statement (1) in the lemma.  

(2) If $b\notin\angle b_I(X)$, then there is an $i\in I$ such that $\langle a_i,b\rangle\le 0$.  Then a sequence $|\ch^{\alpha_i^I(X)}(\la_{(b,0)}(g_\ell))|=|g_\ell|^{\langle a_i,b\rangle}$ $(\ell=1,2,\dots)$ does not converge for any sequence $g_\ell$ $(\ell=1,2,\dots)$ approaching $0$.  This together with \eqref{limit1} and \eqref{limit2} implies statement (2) in the lemma. 
\end{proof}

\begin{lemm} \label{lemm:XI}
$X_I$ is one point for any $I\in\Si^{(n)}(X)$, more generally, $X_I$  is connected for any $I\in\SiX$. 
\end{lemm}

\begin{proof}
We treat the case when $|I|=n$ first.  Let $p$ be any point in $X_I$ and let $b$ be any vector in the interior of $\angle b_I(X)$.  Then  $\lim_{g\to 0}\la_{(b,0)}(g)(\iota)=p$ by Lemma~\ref{lemm:limit} (1), where the left hand side is independent of the choice of $p$.  This shows that $X_I$ is one point.     

A similar idea works for the general case.  By definition, $X$ is covered by $O_p$'s $(p\in X^{(\C^*)^n})$ and each $O_p$ is equivariantly diffeomorphic to a direct sum of complex one-dimensional smooth representation spaces of $(\C^*)^n$, say $V_p$.  Therefore $X_I\cap O_p$ is equivariantly diffeomorphic to an invariant linear subspace of $V_p$ (unless the intersection is empty) and hence each connected component of $X_I$ must have a fixed point.  

Let $F$ be any connected component of $X_I$ and let $p$ be a fixed point in $F$.  Then $p=X_{\tilde I}$ for some $\tilde I\in \Si^{(n)}(X)$ with $I\subset \tilde I$.  Let $b$ be an element in the interior of $\angle b_I(X)$.  Then Lemma~\ref{lemm:limit} (1) applied with $\tilde I$ as $I$ and $I$ as $K$ says that  $\lim_{g\to 0}\la_{(b,0)}(g)(\iota)=:x$ exists in $O_p\cap X_I$, in fact, in $O_p\cap F$ because $O_p\cap X_I$ is connected and $p\in F$.  Therefore $x\in F$.  However the limit point $x$ is independent of the choice of the fixed point $p$ and hence $x$ is contained in every connected component of $X_I$.  This means that $X_I$ is connected.            
\end{proof}

\begin{lemm}
$\bigcup_{I\in\SiX}\angle b_I(X)=\R^n$ and $\angle b_I(X)\cap \angle b_J(X)=\angle b_{I\cap J}(X)$ for any $I,J\in \SiX$.  
\end{lemm}

\begin{proof}
Suppose $\bigcup_{I\in\SiX}\angle b_I(X)\not=\R^n$ and let $b\in\R^n\backslash \bigcup_{I\in\SiX}\angle b_I(X)$.  Since $b\notin \angle b_I(X)$ for any $I\in\Si^{(n)}(X)$ and $X=\bigcup_{p\in X^{(\C^*)^n}}O_p$, it follows from Lemma~\ref{lemm:limit} (2) that $\la_{(b,0)}(g_\ell)(\iota)$ does not converge in $X$ for any sequence $g_\ell$ in $\C^*$ $(\ell=1,2,\dots)$ approaching $0$.  This contradicts the compactness of $X$.  Therefore the former statement in the lemma is proven. 

As for the latter statement in the lemma, it suffices to prove 
\begin{equation} \label{subset}
\angle b_I(X)\cap \angle b_J(X)\subset \angle b_{I\cap J}(X)
\end{equation}
because the opposite inclusion is obvious.  We shall prove \eqref{subset} when $I,J\in\Si^{(n)}(X)$ first.  Suppose that \eqref{subset} does not hold for those $I$ and $J$.  Then there is an element $b\in \angle b_I(X)\cap \angle b_J(X)$ which is not contained in $\angle b_{I\cap J}(X)$.  This means that when we write $b=\sum_{i\in I}r_i b_i(X)$ with $r_i\ge 0$, then there is $i\in I\backslash J$ with $r_i>0$. Similarly, when we write $b=\sum_{j\in J}s_j b_j(X)$ with $s_j\ge 0$, then there is $j\in J\backslash I$ with $s_j>0$.  It follows from Lemma~\ref{lemm:limit} (1) that $\lim_{g\to 0}\la_{(b,0)}(g)(\iota)$ lies in $O_p\cap X_i$ as well as in $O_q\cap X_j$ where $p=X_I$ and $q=X_J$.  In particular $O_p\cap X_i\cap X_j\not=\emptyset$.  However, since $j\notin I$, we have $O_p\cap X_j=\emptyset$ by Lemma~\ref{lemm:lindep}.  This is a contradiction and hence \eqref{subset} holds for $I,J\in\Si^{(n)}(X)$.   

Suppose that $I,J\in \SiX$ may not be in $\Si^{(n)}(X)$.  Then there are $\tilde I, \tilde J\in \Si^{(n)}(X)$ such that $I\subset \tilde I$ and $J\subset \tilde J$ since $\SiX$ is pure.  Let $b\in\angle b_I(X)\cap \angle b_J(X)$.  Since $\angle b_I(X)\subset \angle b_{\tilde I}(X)$, $\angle b_J(X)\subset \angle b_{\tilde J}(X)$ and $\angle b_{\tilde I}(X)\cap \angle b_{\tilde J}(X)=\angle b_{\tilde I \cap\tilde J}(X)$ as proved above, $b$ sits in $\angle b_{\tilde I\cap \tilde J}(X)$. 
Therefore 
\begin{equation} \label{b}
\text{$b=\sum_{k\in \tilde I\cap \tilde J}r_kb_k(X)$\quad  with $r_k\ge 0$.}
\end{equation}
Since $b\in \angle b_I\subset \angle b_{\tilde I}$ and the expression of $b$ as a linear combination of $b_i$'s for $i\in \tilde I$ is unique, $b$ must be a linear combination of $b_i$'s for $i\in I$.  On the other hand, we have the expression \eqref{b} and $\tilde I\cap \tilde J\subset \tilde I$.  It follows that $\{ k\in \tilde I\cap\tilde J\mid r_k\not=0\} \subset I$.  Similarly $\{ k\in \tilde I\cap\tilde J\mid r_k\not=0\} \subset J$.   Therefore 
\[
\{ k\in \tilde I\cap\tilde J\mid r_k\not=0\} \subset I \cap J.
\]
This together with \eqref{b} shows that $b\in \angle b_{I\cap J}(X)$, proving \eqref{subset}.   
\end{proof}

Motivated by the observations above, we make the following definition. 

\begin{defi}
Let $\Si$ be an abstract finite simplicial complex of dimension $n-1$ (with the empty set $\emptyset$ added)  and let 
\[
\beta\colon \Si^{(1)} \to \MR^n=\RZRn
\]
where $\Si^{(1)}$ denotes the vertex set of $\Si$.  We abbreviate an element $\{i\}\in\Si^{(1)}$ as $i$ and $\beta(\{i\})$ as $\beta_i$ and express $\beta_i=(b_i+\sqrt{-1}c_i,v_i)\in\RZRn$.  Then the pair $(\Si,\beta)$ is called a \emph{(simplicial) topological fan} of dimension $n$ if the following are satisfied.
\begin{enumerate}
\item  $b_i$'s for $i\in I$ are linearly independent whenever $I\in \Si$, and $\angle \b_I\cap \angle \b_J=\angle \b_{I\cap J}$ for any $I,J\in \Si$.
(In short, the collection of cones $\angle b_I$ for all $I\in \Si$ is an ordinary simplicial fan although $\b_i$'s are not necessarily in $\Z^n$.) 
\item Each $\v_i$ is primitive and $v_i$'s for $i\in I$ are linearly independent (over $\R$) whenever $I\in \Si$.  
\end{enumerate} 
We say that a topological fan $\De$ of dimension $n$ is \emph{complete} if $\bigcup_{I\in \Si}\angle \b_I=\R^n$ and \emph{non-singular} if the $\v_i$'s for $i\in I$ form a part of a $\Z$-basis of $\Z^n$ whenever $I\in\Si$.  
\end{defi}

\begin{rema}
When $\b_i=\v_i$ and $c_i=0$ for every $i$, $\De$ can be thought of as an ordinary simplicial fan. The notion of completeness and non-singularity above generalizes that for an ordinary simplicial fan.    
\end{rema}

\begin{defi}
Let $\De=(\Si,\beta)$ and $\De'=(\Si',\beta')$ be topological fans.  
\begin{enumerate}
\item $\De$ and $\De'$ are \emph{equivalent}, denoted  $[\De]=[\De']$, if there is an isomorphism $\sigma\colon \Si\to \Si'$ such that $\beta'_{\sigma(i)}=\beta_i$ for any $i\in\Si^{(1)}$.  
\item $\De$ and $\De'$ are \emph{D-equivalent}, denoted $[\De]_D=[\De']_D,$ if there is an isomorphism $\sigma\colon \Si\to \Si'$ such that $\beta'_{\sigma(i)}=\beta_i$ or ${\beta_i}\mu_0$ for any $i\in\Si^{(1)}$, where $\mu_0$ is the element in Lemma~\ref{lemm:repclass}.  
\item $\De$ and $\De'$ are \emph{H-equivalent}, denoted $[\De]_H=[\De']_H$, if there are an isomorphism $\sigma\colon \Si\to \Si'$  and $\mu_i\in \MS$ for each $i\in\Si^{(1)}$ such that $\beta'_{\sigma(i)}=\beta_i\mu_i$ for any $i\in\Si^{(1)}$, where $\MS$ is the subset of $\MR$ in Lemma~\ref{lemm:repclass}. 
\end{enumerate}
\end{defi}

\begin{rema}
One can form a collection of cones in $\MR^n=\RZRn$ from a topological fan.  Then two topological fans $\De$ and $\De'$ are equivalent if and only if the collection of cones derived from $\De$ agrees with that from $\De'$. 

In toric geometry, a morphism between two fans is defined to be an isomorphism of ambient spaces taking one fan to the other. 
One can define morphisms between topological fans similarly.
\end{rema}

For a topological toric manifold $X$ of dimension $2n$, we associated the simplicial complex $\SiX$ of dimension $n-1$ and deduced the map $\beta(X)\colon \Si^{(1)}(X)\to \MR^n$. The observations above show the following.

\begin{lemm}
$\De(X):=(\SiX,\beta(X))$ is a complete non-singular topological fan of dimension $n$. 
\end{lemm}

We gave $X$ the orientation induced from the canonical orientation on the open dense orbit.  However, there is no canonical choice of orientations on the characteristic submanifolds of $X$.  We say that two omnioriented topological toric manifolds are \emph{isomorphic} if there is an equivariant diffeomorphism between them preserving the omniorientations.  An equivariant diffeomorphism or equivariant homeomorphism between omnioriented topological toric manifolds preserves their characteristic submanifolds but does not necessarily preserve the orientations on their characteristic submanifolds.  The following lemma follows from Lemma~\ref{lemm:repclass}. 

\begin{lemm} \label{lemm:iso}
If omnioriented topological toric manifolds $X$ and $Y$ are isomorphic {\rm (}resp. equivariantly diffeomorphic or equivariantly homeomorphic{\rm)}, then $[\De(X)]=[\De(Y)]$ {\rm (}resp. $[\De(X)]_D=[\De(Y)]_D$ or $[\De(X)]_H=[\De(Y)]_H${\rm)}. 
\end{lemm}

\section{Construction of a topological toric manifold $X(\De)$} \label{sect:4}

In Section~\ref{sect:3}, we associated a complete non-singular topological fan $\De(X)$ to an omnioriented topological toric manifold $X$.  In this section and the subsequent two sections, we will define a reverse correspondence, namely, we associate an omnioriented topological toric manifold $X(\De)$ to a complete non-singular topological fan $\De$.  

In toric geometry, there are two ways to associate a toric manifold with a complete non-singular fan: one is the gluing construction described in \cite{ewal96}, \cite{fult93} or \cite{oda88} and the other is the quotient construction where many people are involved in the development, see \cite{cox95} and \cite{cox97}.  Although both constructions work in our setting, we adopt the quotient construction since the quotient construction seems easier to understand.  

Let $\De=(\Si,\beta)$ be a complete non-singular topological fan of dimension $n$ where we take the vertex set $\Si^{(1)}$ as $[m]=\{1,2,\dots,m\}$.  For $I\subset [m]$, we set 
\begin{equation} \label{U(I)}
U(I):=\{(z_1,\dots,z_m)\in \C^m\mid z_i\not=0 \ \text{for $\forall i\notin I$}\}.
\end{equation}
Note that $U(I)\cap U(J)=U(I\cap J)$ for any $I,J\in [m]$ and $U(I)\subset U(J)$ if and only if $I\subset J$.  We define 
\[
U(\Sigma):=\bigcup_{I\in \Si}U(I).
\]

\begin{rema}
It is known (and easy to see) that $U(\Si)$ is the complement of the union of coordinate subspaces 
\[
Z:=\bigcup_{J\notin \Si}\{(z_1,\dots,z_m)\in \C^m\mid z_j=0 \ \text{for $\forall j\in J$}\},
\]
that is, $U(\Si):=\C^m\backslash Z$. 
\end{rema}

Let 
\[
\la \colon (\C^*)^m\to (\C^*)^n
\]
be the homomorphism defined by
\begin{equation} \label{defl}
\la(h_1,\dots,h_m):=\prod_{k=1}^m\la_{\beta_k}(h_k).
\end{equation}

\begin{lemm} \label{lemm:kerl}
$\la$ is surjective and  
\begin{equation} \label{kerl}
\Ker\la=\{(h_1,\dots,h_m)\in (\C^*)^m
\mid \prod_{k=1}^mh_k^{\langle \alpha,\beta_k\rangle}=1\quad \text{for any $\alpha\in \MR^n$}\}.
\end{equation}
Using $\{\alpha^I_i\}_{i\in I}$ for $I\in \Si^{(n)}$, which is dual to $\{\beta_i\}_{i\in I}$, we have  
\begin{equation} \label{aibk}
\Ker\la=\{(h_1,\dots,h_m)\in (\C^*)^m \mid h_i\prod_{k\notin I}h_k^{\langle \alpha^I_i,\beta_k\rangle}=1\quad \text{for any $i\in I$}\}.
\end{equation}
\end{lemm}

\begin{proof}
It follows from Lemma~\ref{lemm:chlalach} (1) that $\la(h_1,\dots,h_m)=\prod_{k=1}^m\la_{\beta_k}(h_k)$ is trivial if and only if  
\begin{equation*} \label{abuv}
\ch^{\alpha}\big(\prod_{k=1}^m\la_{\beta_k}(h_k)\big)=\prod_{k=1}^mh_k^{\langle \alpha,\beta_k\rangle}=1\quad\text{for any $\alpha\in \MR^n$},
\end{equation*}
proving \eqref{kerl}.  
As for \eqref{aibk}, since $\bigoplus_{i\in I}\chi^{\alpha^I_i}$ is an automorphism of $(\C^*)^n$ by Lemma~\ref{lemm:comp}, $\la(h_1,\dots,h_m)=\prod_{k=1}^m\la_{\beta_k}(h_k)$ is trivial if and only if  
\[
\ch^{\alpha^I_i}\big(\prod_{k=1}^m\la_{\beta_k}(h_k)\big)=\prod_{k=1}^mh_k^{\langle \alpha^I_i,\beta_k\rangle}=
h_i\prod_{k\notin I}h_k^{\langle \alpha^I_i,\beta_k\rangle}=1\quad\text{for any $i\in I$},
\]
proving \eqref{aibk}.  
\end{proof}

We define 
\[
X(\De):=U(\Sigma)/\Ker\la=\bigcup_{I\in\Si}U(I)/\Ker\la.
\]
Since the natural action of $(\C^*)^m$ on $\C^m$ leaves the subset $U(\Si)$ of $\C^m$ invariant, it induces an effective action of $(\C^*)^m/\Ker\la$ on $X(\De)$ having an open dense orbit and only finitely many orbits.  Since $\la$ is surjective, $(\C^*)^m/\Ker\la$ can be identified with $(\C^*)^n$ via $\la$, so we think of $X(\De)$ as a topological space with this action of $(\C^*)^n$.  In the next two sections, we will prove that $X(\De)$ is actually a topological toric manifold.

\section{Local properties of $X(\De)$} \label{sect:5}

For $I\in\Si^{(n)}$, we denote by $\C^I$ the affine space $\C^n$ with coordinates indexed by the elements in $I$ and define a continuous map 
\[
\tilde\varphi_I\colon U(I)\to \C^I
\]
by 
\begin{equation} \label{phiI}
\tilde\varphi_I(z_1,\dots,z_m):=(\prod_{k=1}^mz_k^{\langle \alpha^I_i,\beta_k\rangle})_{i\in I}=(z_i\prod_{k\notin I}z_k^{\langle \alpha^I_i,\beta_k\rangle})_{i\in I}
\end{equation}
where $z_k\not=0$ for $k\notin I$ by the definition of $U(I)$. 

\medskip
\noindent
{\bf Claim 1.} $\tilde\varphi_I$ is invariant under the action of $\Ker\la$.

\begin{proof}
Let $(h_1,\dots,h_m)\in\Ker\la$.  Then it follows from \eqref{phiI} and \eqref{kerl} that 
\[
\begin{split}
\tilde\varphi_I(h_1z_1,\dots,h_mz_m)&=\big(\prod_{k=1}^m(h_kz_k)^{\langle \alpha^I_i,\beta_k\rangle}\big)_{i\in I}=\big(\prod_{k=1}^mh_k^{\langle \alpha^I_i,\beta_k\rangle}\prod_{k=1}^mz_k^{\langle \alpha^I_i,\beta_k\rangle}\big)_{i\in I}\\
&=\big(\prod_{k=1}^mz_k^{\langle \alpha^I_i,\beta_k\rangle}\big)_{i\in I}=\tilde\varphi_I(z_1,\dots,z_m), 
\end{split}
\]
proving the claim. 
\end{proof}

By Claim 1 above, $\tilde\varphi_I$ induces a continuous map 
$$\varphi_I\colon U(I)/\Ker\la\to\C^I.$$

\medskip
\noindent
{\bf Claim 2.} $\varphi_I$ is a homeomorphism. 

\begin{proof}
By the definition of $U(I)$, $z_k\not=0$ for $k\notin I$ and $z_i$ for $i\in I$ is an arbitrary complex number, so $\varphi_I$ is surjective.  

The injectivity of $\varphi_I$ is as follows.  Suppose that 
\[
\tilde\varphi_I(z_1,\dots,z_m)=\tilde\varphi_I(\zeta_1,\dots,\zeta_m)
\]
for $(z_1,\dots,z_m), (\zeta_1,\dots,\zeta_m)\in U(I)$.  Then it follows from \eqref{phiI} that 
\begin{equation} \label{ziwi}
z_i\prod_{k\notin I}z_k^{\langle \alpha^I_i,\beta_k\rangle}
=\zeta_i\prod_{k\notin I}\zeta_k^{\langle \alpha^I_i,\beta_k\rangle}
\end{equation}
for each $i\in I$.  We define an element $(h_1,\dots,h_m)\in (\C^*)^m$ by 
\begin{equation} \label{hkhi}
h_k:=z_k/\zeta_k\ \text{for $k\notin I$}\quad\text{and}\quad h_i:=\prod_{k\notin I}h_k^{-\langle \alpha^I_i,\beta_k\rangle}\ \text{for $i\in I$}.
\end{equation}
Then $(h_1,\dots,h_m)\in \Ker\la$ by \eqref{aibk} and $(z_1,\dots,z_m)=(h_1\zeta_1,\dots,h_m\zeta_m)$ by \eqref{ziwi} and \eqref{hkhi}, proving the injectivity of $\varphi_I$. 

Finally, if $\tilde f_I$ is a map sending $(w_i)_{i\in I}\in \C^I$ to $(z_1,\dots,z_m)\in U(I)$ where $z_i=w_i$ for $i\in I$ and $z_k=1$ for $k\notin I$, then  the composition $\tilde\varphi_I\circ \tilde f_I$ is the identity on $\C^I$ by \eqref{phiI}.  Since $\tilde f_I$ is continuous, this shows that $\varphi_I^{-1}$ is also continuous. 

This completes the proof of the claim.  
\end{proof}

Remember that $X(\De)$ has an action of $(\C^*)^n$.  We shall observe this action through the local chart $\varphi_I\colon U(I)/\Ker\la\to \C^I$. We note that the underlying space of the representation space $V(\bigoplus_{i\in I}\ch^{\alpha^I_i})$ can naturally be identified with $\C^I$.    

\begin{lemm} \label{lemm:AIUI}
The homeomorphism 
$$\varphi_I\colon U(I)/\Ker\la\to V(\bigoplus_{i\in I}\ch^{\alpha^I_i})$$ 
is $(\C^*)^n$-equivariant.
\end{lemm}

\begin{proof}
To $g\in (\C^*)^n$ we associate an element $h=(h_1,\dots,h_m)\in(\C^*)^m$ defined by 
\[
h_i=\begin{cases} \ch^{\alpha^I_i}(g) \quad&\text{for $i\in I$},\\
1 \quad&\text{for $i\notin I$.}\end{cases}
\]
Then 
$$\la(h)=\prod_{i=1}^m\la_{\beta_i}(h_i)=\prod_{i\in I}\la_{\beta_i}(\chi^{\alpha^I_i}(g))=g$$
where the last identity follows from Lemma~\ref{lemm:comp}.   
Let $z=(z_1,\dots,z_m)\in U(I)$ where $z_i=1$ for $i\notin I$.  Then $\varphi_I(z)=(z_i)_{i\in I}$ and $\varphi_I(hz)=(h_iz_i)_{i\in I}=(\ch^{\alpha^I_i}(g)z_i)_{i\in I}$.  This shows that $\varphi_I$ is $(\C^*)^n$-equivariant.  
\end{proof}

\begin{lemm} \label{lemm:transition}
The transition functions of the local charts $\{(U(I)/\Ker\la,\varphi_I)\}$ for $X(\De)$ are smooth.  
\end{lemm}  

\begin{proof}
Let $J$ be another element of $\Si^{(n)}$ such that $\{v_j\mid j\in J\}$ is a $\Z$-basis of $\Z^n$.  Since 
$\varphi_J\circ \varphi_I^{-1}=\tilde\varphi_J\circ \tilde f_I$, it follows from \eqref{phiI} that the $j$-component of $\varphi_J(\varphi_I^{-1}(w_i)_{i\in I})$ for $(w_i)_{i\in I}\in \varphi_I(U(I)\cap U(J)\textcolor{blue}{)}\subset \C^I$ is given by 
\begin{equation} \label{tran}
\prod_{i\in I}w_i^{\langle \alpha^J_j,\beta_i\rangle}
\end{equation}
which is smooth since $w_i\not=0$ for $i\in I\backslash J$ and $\langle \alpha^J_j,\beta_i\rangle=\delta_{ji}\1$ for $i\in J$. 
\end{proof}

We express $\beta_i=(b_i+\sqrt{-1}c_i,v_i)\in\RZRn=\MR^n$ as before.  

\begin{lemm} \label{lemm:conjugate}
If $c_i=0$ for any $i$, then the transition function \eqref{tran} is equivariant with respect to the complex conjugation on $\C^n$, so that the complex conjugation on $\C^n$ induces an involution, called the \emph{conjugation}, on $X(\De)$ with $n$-dimensional fixed point set. 
\end{lemm}  

\begin{proof}
Since $c_i=0$ for any $i$, the $c$-component of $\langle \alpha^J_j,\beta_i\rangle$ in \eqref{tran} is zero.  Therefore the lemma follows from Lemma~\ref{lemm:gmu} (2).   
\end{proof}

\begin{rema}
The fixed point set of the conjugation on $X(\De)$ in Lemma~\ref{lemm:conjugate} has the restricted action of $(\R^*)^n$ which has an open dense orbit isomorphic to $(\R^*)^n$ and the number of orbits is finite.  Such a manifold is called a \emph{real topological toric manifold} and discussed in Section~\ref{sect:12}.
\end{rema}  

If $\b_i=\v_i$ and $c_i=0$ for any $i$, then $X(\De)$ is a toric manifold and the transition function in \eqref{tran} is a Laurent monomial in $w_i$'s.  More generally, it follows from Lemma~\ref{lemm:gmu} (3) that the transition function in \eqref{tran} is a Laurent monomial in $w_i$'s and $\bar w_i$'s if $\b_i$ is an integral vector congruent to $\v_i$ modulo 2 and $c_i=0$ for any $i$. 

\begin{exam}
Here is an example of a topological toric manifold of dimension $4$ which is not a toric manifold.  Let $\Si$ be an abstract simplicial complex defined by 
\[
\Si:=\{\emptyset, \{1\}, \{2\}, \{3\}, \{4\}, \{1,2\}, \{2,3\}, \{3,4\}, \{4,1\}\}.
\]
Let $e_1, e_2$ denote the standard basis of $\R^2$ and define 
\begin{align*}
\beta_1&=(b_1,v_1):=(e_1,e_1), &\beta_2&=(b_2,v_2):=(e_2,e_2)\\
\beta_3&=(b_3,v_3):=(-e_1,-e_1-2e_2), & \beta_4&=(b_4,v_4):=(-e_1-e_2,-e_1-e_2)
\end{align*}
where $c_i$ for $1\le i\le 4$ is understood to be $0$.  Note that $\De=(\Si,\beta)$ is a complete non-singular topological fan of dimension 2.  Note also that $(\Si,b)$ defines an ordinary complete fan while the 2-dimensional cones obtained from $(\Si,v)$ have an overlap and defines a multi-fan, see figure~\ref{fig:fans}.  Note also that $b_i$ is an integral vector congruent to $v_i$ modulo $2$ and $c_i=0$ for any $i$, so that the transition functions in our case should be Laurent monomials in $w_1, w_2, \bar{w}_1$ and $\bar{w}_2$.  In fact, they are given explicitly below.  

\begin{center}
\begin{figure}[h]
	\includegraphics{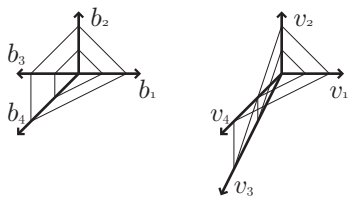}
\caption{vectors $b_i$ and $v_i$}
\label{fig:fans}
\end{figure}
\end{center} 

The dual set of $\{\beta_j\}_{j\in J}$ for $J\in\Si^{(2)}$ is given as follows.
\begin{align*}
\alpha_1^J&=(e_1,e_1) , & \alpha_2^J&=(e_2,e_2) & \text{for $J=\{1,2\}$,}\\ 
\alpha_2^J&=(e_2,-2e_1+e_2) , & \alpha_3^J&=(-e_1,-e_1) & \text{for $J=\{2,3\}$,}\\ 
\alpha_3^J&=(-e_1+e_2,e_1-e_2) , & \alpha_4^J&=(-e_2,-2e_1+e_2) & \text{for $J=\{3,4\}$,}\\ 
\alpha_4^J&=(-e_2,-e_2) , & \alpha_1^J&=(e_1-e_2,e_1-e_2) & \text{for $J=\{4,1\}$.}
\end{align*}
Therefore, it follows from \eqref{tran} that $X(\De)$ is obtained by gluing four copies of $\C^2$ corresponding to elements in $\Si^{(2)}$ as follows: 
\[
\begin{CD}
\quad \C^2=\C^{\{3,2\}}\qquad (w_1^{-1}, w_1^{-1}\bar{w}_1w_2) @<<< \qquad (w_1,w_2) \qquad \C^{\{1,2\}}=\C^2\\
\qquad\qquad @V f_1VV  @VVV \qquad \ \quad \ \\
\quad \C^2=\C^{\{3,4\}}\quad (\bar{w}_1^{-1}\bar{w}_2, \bar{w}_1w_1^{-1}\bar{w}_2^{-1})@<f_2<< \quad (w_1w_2^{-1},w_2^{-1}) \quad \C^{\{1,4\}}=\C^2\\
\end{CD}
\]
where $f_1(u_1,u_2)=(u_1\bar{u}_2,\bar{u}_2^{-1})$, $f_2(v_1,v_2)=(\bar{v}_1^{-1}, \bar{v}_1v_1^{-1}v_2)$ and the horizontal maps above glue $\C^2$ along $\C^*\times \C$ while the vertical ones glue $\C^2$ along $\C\times\C^*$. 

One can check that $X(\De)$ is a closed smooth manifold (this is true for an arbitrary complete non-singular topological fan $\De$ as is seen in the next section).  It is simply connected by Proposition~\ref{prop:simply}, of dimension 4 and admits a smooth effective action of $(S^1)^2$.  One can also check that our $X(\De)$ has the same cohomology ring as $\C P^2\#\C P^2$ using Proposition~\ref{prop:cohom} in Section~\ref{sect:8}.  Therefore it follows from the classification result on closed smooth 4-manifolds with effective smooth action of $(S^1)^2$ (see \cite{or-ra70}) that our $X(\De)$ is diffeomorphic to $\C P^2\#\C P^2$.  As is well-known, $\C P^2\#\C P^2$ is not (the underlying manifold of) a toric manifold, in fact, it does not admit\textcolor{red}{s} an almost complex structure.   
\end{exam}

\section{Global properties of $X(\De)$} \label{sect:6}

In this section, we establish that $X(\De)$ is a topological toric manifold and study the tangent bundle of $X(\De)$.  

\begin{lemm} \label{lemm:haus}
$X(\De)$ is Hausdorff.
\end{lemm}

\begin{proof}
Let $[x]$ and $[y]$ be points in $X(\De )$ represented by $x, y\in U(\Si)$.  If $[x]$ and $[y]$ are contained in a same local chart, they can be separated by open neighborhoods since the chart is homeomorphic to $\C^n$.  Suppose $x$ is contained in $U(I)\setminus U(J)$ and $y$ is contained in $U(J)\setminus U(I)$.  Since $\dim \angle b_I\cap\angle b_J\le n-1$, there is an element $a \in \R ^n$ such that
\begin{equation*}
\begin{split}
&H^+ := \{ b \in \R ^n \mid \langle a,b \rangle \geq 0 \} \supset \angle b_I, \quad 	H^- := \{ b \in \R ^n \mid \langle a,b \rangle \leq 0 \} \supset \angle b_J,\\ 
&\text{and}\quad H^+ \cap \angle b_J = H^- \cap \angle b_I = \angle b_I \cap \angle b_J.
\end{split} 
\end{equation*}

We define $f_1\colon U(I) \to \R _{\geq 0}$ and $f_2\colon U(J)\to \R _{\geq 0}$ by 
\begin{equation*}
	f_1(z_1,\dots,z_m):=\prod _{k=1}^m |z_k|^{\langle a,b_k\rangle}\quad\text{and}\quad f_2(z_1,\dots,z_m):=\prod _{k=1}^m |z_k|^{-\langle a,b_k\rangle}.
\end{equation*}
Both functions are well-defined (because $\langle a,b_i\rangle \ge 0$ for $i\in I$ and $-\langle a,b_j\rangle \ge 0$ for $j\in J$) and invariant under the action of $\Ker\la$ by \eqref{kerl}.   Moreover $f_1 =f_2^{-1}$ on $U(I) \cap U(J)$. Therefore, $f_1$ and $f_2$ define a continuous function 
\begin{equation*}
	f_1\cup f_2^{-1} : (U(I)\cup U(J))/\Ker \la \to [0,\infty ].
\end{equation*}

Since $x=(x_1,\dots,x_m)$ is contained in $U(I)\backslash U(J)$, $x_i = 0$ for some $i \in I\setminus J$ by \eqref{U(I)} while $\langle a,b_i \rangle $ is positive for any $i\in I\setminus J$ by the choice of $a$. Thus $f_1(x)=0$ and hence $(f_1\cup f_2^{-1})([x]) =0$. Similarly, since $f_2(y)=0$, we have $(f_1\cup f_2^{-1})([y])=\infty$. Therefore $X(\De )$ is Hausdorff. 
\end{proof}

\begin{lemm} \label{lemm:compact}
$X(\De)$ is compact.
\end{lemm}

\begin{proof}
We shall show that $X (\De )$ is covered by finitely many compact subsets. By the definition of $X(\De )$, $(\C ^*)^n$ is naturally embedded as the open dense orbit. The embedding $(\C ^*)^n \hookrightarrow X(\De )$ is given by 
\begin{equation} \label{embe}
	(\C^*)^n\ni g  \stackrel{\chi_I}\longrightarrow ( \ch^{\alpha^I_i}(g))_{i\in I}\in V(\bigoplus_{i\in I}\ch^{\alpha^I_i}) \stackrel{\varphi_I^{-1}}\longrightarrow U(I)/\Ker\la \subset X(\De),
\end{equation}
where $I\in\Si^{(n)}$ and the embedding does not depend on the choice of $I$ as is shown in the proof of Lemma~\ref{lemm:AIUI}.  We consider a  subset 
\begin{equation} \label{CI}
	C_I:=\{ g \in (\C ^*)^n \mid  |\ch^{\alpha^I_i}(g)|=\prod_{k=1}^n|g_k|^{a^{Ik}_i} \leq 1 \text{ for all } i \in I\}, 
\end{equation}
where $g=(g_1,\dots,g_n)$ and  $a^I_i=(a_i^{I1},\dots,a_i^{In}) \in \R^n$ is the real part of the first component of $\alpha^I_i\in \MR^n=\RZRn$.  We also consider a surjective homomorphism
\begin{equation} \label{-log}
	-\log|\ | : g=(g_1,\dots ,g_n) \to -\log |g|=(-\log |g_1|,\dots , -\log |g_n|)
\end{equation}
from $(\C ^*)^n$ to $\R ^n$.  It follows from \eqref{CI} and \eqref{-log} that $\langle a^I_i,-\log |g|\rangle \ge 0$ if and only if $g\in C_I$. Hence 
\begin{equation} \label{logCI}
\text{$C_I=(-\log |\ |)^{-1}(\angle b_I)$.}
\end{equation}
Since $\De $ is complete, this shows that 
\begin{equation} \label{C*de}
	(\C ^*)^n = \bigcup _{I \in \Si ^{(n)}}C_I\quad\text{and hence} \quad X(\De)=\bigcup_{I\in \Si^{(n)}}\overline{\varphi_I^{-1}(\chi_I(C_I))}
\end{equation}
where $\varphi_I^{-1}\circ \chi_I$ is the embedding of $(\C^*)^n$ into $X(\De)$ in \eqref{embe}.  The closure $\overline{\varphi_I^{-1}(\chi_I(C_I))}$ is a compact subset of $U(I)/\ker \la$, that is homeomorphic to the polydisk $\{(w_i)_{i\in I}\in \C^I \mid |w_i|\le 1 \text{ for $\forall i\in I$}\}$ via $\varphi_I$ by \eqref{CI}. Hence $X (\De )$ is compact. 
\end{proof}

\begin{coro} \label{coro:XDelta}
$X(\De)$ with the local charts $\{(U(I)/\Ker\la,\varphi_I)\}$ is a topological toric manifold.
\end{coro}

\begin{proof}
By Lemmas~\ref{lemm:haus} and~\ref{lemm:compact} $X(\De)$ is Hausdorff and compact and then Lemmas~\ref{lemm:AIUI} and~\ref{lemm:transition} show that $X(\De)$ is a topological toric manifold. 
\end{proof}

The topological toric manifold $X(\De)$ associated with a topological fan $\De=(\Si,\beta)$ has a canonical omniorientation.  In fact, the image of 
$$U(\Si)_i:=U(\Si)\cap \{(z_1,\dots,z_m)\in\C^m\mid z_i=0\}$$ 
by the quotient map $U(\Si)\to U(\Si)/\Ker\la=X(\De)$ is a characteristic submanifold $X(\De)_i$ of $X(\De)$ and the canonical normal orientation of  $U(\Si)_i$ to $U(\Si)$ induced from the $i$-th factor of $\C^m$ defines a normal orientation of $X(\De)_i$ to $X(\De)$ through the projection map.  This together with the orientation on $X(\De)$ (induced from the orientation on the open dense orbit) determines an orientation on $X(\De)_i$, thus an omniorientation on $X(\De)$ is assigned.   

Note that the identity \eqref{norm} holds for $\beta_i:=\beta(i)$ with respect to this omniorientation on $X(\De)$.  It is also clear that the simplicial complex $\Si(X(\De))$ associated with $X(\De)$ is the given $\Si$.  For later use we will record this fact as a lemma.

\begin{lemm} \label{lemm:cano}
The topological fan associated with the omnioriented topological toric manifold $X(\De)$ is the given topological fan $\De$.
\end{lemm}  

Let $\pi_i\colon (\C^*)^m \to \C^*$ be the projection onto the $i$-th factor of $(\C^*)^m$ and denote by $V(\pi_i)$ the complex one-dimensional representation space of $\pi_i$.  We consider the diagonal action of $(\C^*)^m$ on $U(\Si)\times V(\pi_i)$, i.e. 
\[
\begin{split}
&(g_1,\dots,g_m)((z_1,\dots,z_m),w)=((g_1z_1,\dots,g_mz_m),g_iw) \\
&\qquad\text{for $(g_1,\dots,g_m)\in (\C^*)^m,\ (z_1,\dots,z_m)\in U(\Si),\ w\in V(\pi_i)$}   
\end{split}
\]
and take the quotient by $\Ker\la$.  Since the action of $\Ker\la$ on $U(\Si)$ is free and the quotient by $\Ker\la$ has the residual action of $(\C^*)^m/\Ker\la=(\C^*)^n$, the first projection 
\[
(U(\Si)\times V(\pi_i))/\Ker\la \to U(\Si)/\Ker\la=X(\De)
\]
defines a complex $(\C^*)^n$-line bundle over $X(\De)$, which we denote by $L_i$.  On the other hand, since $X(\De)$ has the smooth action of $(\C^*)^n$, the tangent bundle $\tau X(\De)$ of $X(\De)$ is a $(\C^*)^n$-vector bundle.  

\begin{theo} \label{theo:bundle}
There is a canonical short exact sequence of $(\C^*)^n$-vector bundles
\[
 0\to \uC^{m-n}\to \bigoplus_{i=1}^mL_i  \to \tau X(\De)\to 0
\]
where $\uC^{m-n}$ denotes the product bundle over $X(\De)$ with fiber $\C^{m-n}$ on which the action of $(\C^*)^n$ is trivial.  
\end{theo}

\begin{proof}
Although our proof is essentially same as in \cite[Proposition 4.5]{bu-pa-ra07}, the paper \cite{bu-pa-ra07} treats quasitoric manifolds where the acting group is a compact torus; so we shall give the proof below for convenience sake of the readers.

The differential of the quotient map $q\colon U(\Si)\to U(\Si)/\Ker\la=X(\De)$ induces a surjective bundle homomorphism 
\[
dq\colon \tau U(\Si)=U(\Si)\times \C^m \to \tau X(\De).
\]
The kernel of $dq$ is the tangent bundle along the fibers of the fiber bundle $q\colon U(\Si)\to X(\De)$ and can be identified with the trivial bundle $U(\Si)\times \Lie(\Ker\la)$ where $\Lie(\Ker\la)$ is the Lie algebra of $\Ker\la$.  Therefore we have a short exact sequence of vector bundles 
\begin{equation} \label{short}
0\to U(\Si)\times \Lie(\Ker\la)\stackrel{\rho}\longrightarrow \tau U(\Si)=U(\Si)\times \C^m \stackrel{dq}\longrightarrow \tau X(\De)\to 0
\end{equation}
where the monomorphism $\rho$ above is defined by $\rho(z,\eta)=(z,\eta_z)$ where $\eta_z$ denotes the fundamental vector field at $z$ associated with $\eta\in \Lie(\Ker\la)$.   In fact, the short exact sequence \eqref{short} is equivariant with respect to the natural actions of $(\C^*)^m$ where the action of $(\C^*)^m$ on $\Lie(\Ker\la)$ is trivial, that on $\C^m$ is the coordinatewise multiplication and that on $\tau X(\De)$ descends to the action of $(\C^*)^m/\Ker\la=(\C^*)^n$.  Taking quotients by $\Ker\la$, \eqref{short} reduces to a short exact sequence of $(\C^*)^n$-vector bundles 
\begin{equation} \label{short2}
 0\to X(\De)\times \Lie(\Ker\la)\to (U(\Si)\times \C^m)/\Ker\la \to \tau X(\De)\to 0.
\end{equation}
Here $X(\De)\times \Lie(\Ker\la)=\uC^{m-n}$ and $(U(\Si)\times \C^m)/\Ker\la=\bigoplus_{i=1}^mL_i$, so \eqref{short2} implies the theorem. 
\end{proof}

\section{$X(\De)$ as an $(S^1)^n$-manifold} \label{sect:7}

In this section we study the topology of $X(\De)$ as an $(S^1)^n$-manifold.  Since $X(\De)$ is locally equivariantly diffeomorphic to sum of complex one-dimensional representation spaces, the orbit space $X(\De)/(S^1)^n$ is a manifold with corners, so that faces of $X(\De)/(S^1)^n$ can naturally be defined.  We think of $X(\De )/(S^1)^n$ itself as a face of $X(\De )/(S^1)^n$. 

\begin{lemm} \label{lemm:faces}
All faces of $X(\De )/(S^1)^n$ are contractible and the face poset of $X(\De )/(S^1)^n$ coincides with the inverse poset of $\Si$.
\end{lemm} 

\begin{proof}
We will give a cubical complex structure which is introduced in \cite[Construction 4.8]{bu-pa02} to the orbit space $X(\De )/(S^1)^n$ using the decomposition 
\[
X(\De)/(S^1)^n=\bigcup_{I\in\Si^{(n)}}\overline{\varphi_I^{-1}(\chi_I(C_I))}/(S^1)^n
\]
induced from the decomposition \eqref{C*de}.  
For $I \in \Si ^{(n)}$, we define a map 
$$\psi _I \colon \overline{\varphi_I^{-1}(\chi_I(C_I))}/(S^1)^n \to [0,1]^m$$ 
as follows.  For a point $p$ in $\overline{\varphi_I^{-1}(\chi_I(C_I))}$ we denote its image in the quotient $\overline{\varphi_I^{-1}(\chi_I(C_I))}/(S^1)^n$ by $[p]$.  Then the $k$-th coordinate of $\psi_I([p])$ is defined to be  
\begin{equation*}
	\begin{cases}
		|\varphi_I(p)_k| & \text{ if $k\in I$},\\
		1 & \text{ if $k \notin I$}.
	\end{cases}
\end{equation*}

\medskip
\noindent
{\bf Claim.} $\psi _I = \psi _J$ on $\overline{\varphi_I^{-1}(\chi_I(C_I))}/(S^1)^n \cap \overline{\varphi_J^{-1}(\chi_J(C_J))}/(S^1)^n$ for any $I,J\in\Si^{(n)}$.

\medskip  
It follows from \eqref{logCI} that $-\log|C_I\cap C_J|=\angle b_{I\cap J}$. This means that if $g\in C_I\cap C_J$, then $\langle a^I_i,-\log |g|\rangle=\langle a^J_j,-\log |g|\rangle=0$, equivalently $|g|^{a^I_i}=|g|^{a^J_j}=1$, again equivalently $|\chi^{\alpha^I_i}(g)|=|\chi^{\alpha^J_j}(g)|=1$, for $i\in I\backslash J$ and $j\in J\backslash I$.  This implies that    
\begin{equation} \label{CICJ}
\begin{split}
	&\overline{\varphi_I^{-1}(\chi_I(C_I))}\cap \overline{\varphi_J^{-1}(\chi_J(C_J))}\\
=\ &\{ p \in X(\De )\mid |\varphi_I(p)_i|=|\varphi_J(p)_j|=1 \text{ for any $i\in I \setminus J$ and $j \in J \setminus I$}\} .
\end{split}
\end{equation}
Therefore, if  $k \in I \cap J$, then it follows from \eqref{tran} and \eqref{CICJ} that 
\[
	|\varphi_J(p)_k| = \prod _{i\in I}|\varphi_I(p)_i|^{\langle a^J_k,b_i\rangle}=\prod_{j\in I\cap J}|\varphi_I(p)_j|^{\langle a^J_k,b_j\rangle} =|\varphi_I(p)_k|
\]
for $p \in \overline{\varphi_I^{-1}(\chi_I(C_I))}\cap \overline{\varphi_J^{-1}(\chi_J(C_J))}$, where the last equality above follows from the fact $\langle a^J_k,b_j\rangle=\delta_{kj}$.  This proves the claim. 

It follows from the claim above that the map 
\begin{equation*}
	\psi:=\bigcup _{I \in \Si ^{(n)}}\psi _I : X(\De )/(S^1)^n \to [0,1]^m
\end{equation*}
is well-defined.  Moreover, this map is an into homeomorphism and the image is determined by $\Si$.  Since each $\psi _I\big(\overline{\varphi_I^{-1}(\chi_I(C_I))}/(S^1)^n\big)$ is a cone with the cone point $(1,\dots ,1) \in [0,1]^m$, so is $X(\De)/(S^1)^n$ and hence contractible.  

Through the map $\psi$, a face of $X(\De)/(S^1)^n$ can be described as 
\begin{equation*}
\psi(X(\De )/(S^1)^n) \cap \{ x_j = 0 \mid j \in J\}
\end{equation*}
for some $J \in \Si $ where $x_j$ denotes the $j$-th coordinate of $[0,1]^m$. This shows that not only $X(\De)/(S^1)^n$ is contractible but also all proper faces of $X(\De)/(S^1)^n$ are contractible.  Moreover it shows that the face poset of $X(\De )/(S^1)^n$ coincides with the inverse poset of $\Si$.
\end{proof}

\begin{theo} \label{theo:homeo}
Let $\De=(\Si,\beta)$ be a complete non-singular topological fan of dimension $n$.  Then the $(S^1)^n$-equivariant homeomorphism type of $X(\De)$ does not depend on the first components $b+\sqrt{-1}c$ of $\beta=(b+\sqrt{-1}c,v)\colon \Si^{(1)}\to \MR^n=\RZRn$.  
\end{theo}

\begin{proof}
	For each $I \in \Si ^{(n)}$, we regard $(\R _{\geq 0})^I$ as a closed subset of $\C ^I$. 
	Then, the quotient space $X(\De )/(S^1)^n$ can be regarded as a subset of $X(\De )$ by 
	\begin{equation*}
		X(\De )/(S^1)^n = \bigcup _{I \in \Si ^{(n)}}\varphi _I^{-1}((\R _{\geq 0})^n)
	\end{equation*}
and hence the map 
$$f \colon (S^1)^n \times (X(\De)/(S^1)^n) \to X(\De )$$ 
defined by $f(g,x) = gx$ makes sense.   If  $f(g,x) = f(g',x')$ for $(g,x), (g', x') \in (S^1)^n \times (X(\De )/(S^1)^n)$, then $x=x'$ and $g'g^{-1}x = x$, that is, $g'g^{-1}$ is an element of the isotropy subgroup of $(S^1)^n$ at $x$. Since the isotropy subgroup at $x$ is determined by the face $F$ containing $x$ in its relative interior and the second component $v$ of $\beta$, the $(S^1)^n$-equivariant homeomorphism type does not depend on the first component of $\beta$.
\end{proof}

\begin{rema}
It is unclear that Theorem~\ref{theo:homeo} holds even for the $(S^1)^n$-equivariant \emph{diffeomorphism} type of $X(\De)$ although one can see that the $(S^1)^n$-equivariant diffeomorphism type of $X(\De)$ remain unchanged through a continuous perturbation of the first component $b+\sqrt{-1}c$ of $\beta$ (with $\Si$ and $v$ fixed).   
\end{rema}

\section{Classification of topological toric manifolds} \label{sect:8}

We have associated a complete non-singular topological fan $\De(X)$ to an omnioriented topological toric manifold $X$ in Section~\ref{sect:3}.  Remember that omnioriented topological toric manifolds $X$ and $Y$ are isomorphic if there is an equivariant diffeomorphism preserving the omniorientations and if they are isomorphic, then the topological fans $\De(X)$ and $\De(Y)$ associated with $X$ and $Y$ are equivalent, that is, $[\De(X)]=[\De(Y)]$, see  Lemma~\ref{lemm:iso}.  Conversely we have constructed an omnioriented topological toric manifold $X(\De)$ to a complete non-singular topological fan $\De$ in Sections~\ref{sect:4}, \ref{sect:5} and \ref{sect:6}.   So we have two correspondences in opposite directions: 
\[
\begin{split}
&\{\text{Omnioriented topological toric manifolds of dimension $2n$}\}\\
&\rightleftarrows\ \{\text{Complete, non-singular topological fans of dimension $n$}\}
\end{split}
\]
where we do not distinguish between isomorphic omnioriented topological toric manifolds and also between equivalent topological fans.  We shall denote the correspondence $\rightarrow$ (resp. $\leftarrow$) above by $\MD$ (resp. $\MX$).  

\begin{theo} \label{theo:class}
The correspondences $\MD$ and $\MX$ are inverses to each other, so they are both bijections. 
\end{theo}

\begin{proof}
Lemma~\ref{lemm:cano} says that the composition $\MD\circ\MX$ is the identity.  Therefore, it suffices to prove that $\MD$ is injective.  

Let $X$ be an omnioriented topological toric manifold of dimension $2n$.  As observed in Lemma~\ref{lemm:lindep}, the associated topological fan $\De(X)$ determines the complex $n$-dimensional representation space $V_p$  for each fixed point $p$ in $X$ such that there is an equivariant diffeomorphism to the invariant open neighborhood $O_p$ preserving the omniorientations, where the omniorientation on $V_p$ is the one induced from the complex structure on $V_p$ and that on $O_p$ is the one induced from $X$.  Choose a point $\iota$ in the open dense orbit in $X$.  We may assume that the omniorientation preserving equivariant diffeomorphism $O_p\to V_p$, denoted by $\varphi_p$, sends $\iota$ to the point $1_p:=(1,\dots,1)\in V_p(=\C^n)$ for each $p$ if necessary by composing an automorphism of $V_p$ given by an element of $(\C^*)^n$.  Therefore, the transition function $\varphi_q\circ \varphi_p^{-1}\colon \varphi _p(O_p\cap O_q) \to \varphi _q(O_p\cap O_q)$ for $q\in X^{(\C^*)^n}$ maps $1_p$ to $1_q$.  Since the transition function is equivariant and defined on a dense subset of $V_p$, such a map is unique.  This shows that $X$ is determined by the associated topological fan $\De(X)$, which means the injectivity of $\MD$. 
\end{proof}

\begin{coro} \label{coro:class2}
Two omnioriented topological toric manifolds $X$ and $Y$ are equivariantly diffeomorphic {\rm(}resp. equivariantly homeomorphic{\rm)} if and only if $[\De(X)]_D=[\De(Y)]_D$ {\rm(}resp. $[\De(X)]_H=[\De(Y)]_H${\rm)}. 
\end{coro}

\begin{proof}
The \lq\lq only if" part is proved by Lemma~\ref{lemm:iso}.  Suppose $[\De(X)]_D=[\De(Y)]_D$.  Then, changing the orientations on the characteristic submanifolds of $X$ if necessary, we may assume that $[\De(X)]=[\De(Y)]$.  Then it follows from Theorem~\ref{theo:class} that $X$ (with the changed omniorientation) and $Y$ are isomorphic, in particular they are equivariantly diffeomorphic.

Now suppose $[\De(X)]_H=[\De(Y)]_H$.  Then, changing the local coordinates of $X$ and $Y$ through equivariant homeomorphism, one sees that there is a topological toric manifold $X'$ such that $X'$ is equivariantly homeomorphic to $X$ and $[\De(X')]=[\De(Y)]$ by Lemma~\ref{lemm:repclass}.  Then $X'$ and $Y$ are isomorphic by Theorem~\ref{theo:class} and hence $X$ and $Y$ are equivariantly homeomorphic.      
\end{proof}

Remember that the orbit space $X/(S^1)^n$ of a topological toric manifold $X$ by the restricted action of $(S^1)^n$ is a manifold with corners.  Since we may assume that $X$ is of the form $X(\De)$ by Theorem~\ref{theo:class}, it follows from Lemma~\ref{lemm:faces} that every face of $X/(S^1)^n$ (even $X/(S^1)^n$ itself) is contractible and any intersection of faces is connected unless it is empty.  Therefore the following proposition follows from Corollary 7.8 and Theorem 8.3 in \cite{ma-pa06}. 

\begin{prop} \label{prop:cohom}
Let $X$ be an omnioriented topological toric manifold and let $\De(X)=(\SiX, \beta(X))$ be the complete non-singular topological fan associated with $X$.  We may assume that the vertex set $\Si^{(1)}(X)$ of $\Si(X)$ is $[m]$ and express $\beta(X)(i)=(b_i(X)+\sqrt{-1}c_i(X),v_i(X))\in \RZRn$ for $i\in [m]=\Si^{(1)}(X)$.  Then 
\[
H^*(X;\Z)=\Z[\mu_1,\dots,\mu_m]/\MI
\]
where $\mu_i\in H^2(X;\Z)$ and $\MI$ is the ideal generated by the following two types of elements:
\begin{enumerate}
\item $\prod_{i\in I} \mu_i$ \quad for $I\notin \SiX$,
\item $\sum_{i=1}^m\langle u,v_i(X)\rangle \mu_i$ \quad for any $u\in \Z^n$.
\end{enumerate}
Here $\mu_i$ is actually the Poincar\'e dual of the characteristic submanifold $X_i$ in $X$.    
\end{prop}

\begin{rema} 
Since the homeomorphism type of $X$ does not depend on the $b_i(X)+\sqrt{-1}c_i(X)$'s by Theorem~\ref{theo:homeo}, they should not appear in the expression of $H^*(X;\Z)$ as above.  When $X$ is a toric manifold, Proposition~\ref{prop:cohom} is well-known as Jurkiewicz-Danilov's Theorem (see p.134 in \cite{oda88}).    
\end{rema}

As mentioned in Proposition~\ref{prop:cohom}, $\mu_i$ is the Poincar\'e dual of $X_i$ in $X$.  On the other hand, we may assume $X=X(\De)$ by Theorem~\ref{theo:class}, where $\De=\De(X)$.  Then it is not difficult to see that the complex line bundle $L_i$ over $X(\De)$ in Theorem~\ref{theo:bundle} is an extension of the normal bundle $\nu_i$ of $X_i=X(\De)_i$ as real vector bundles.  This implies that the total Pontrjagin class of $L_i$ is given by $1+\mu_i^2$.  Thus the following proposition follows from Theorem~\ref{theo:bundle}. 

\begin{prop} \label{prop:pont}
Let the situation and the notation be the same as in Proposition~\ref{prop:cohom}.  Then the total Pontrjagin class of $X$ is given by 
$\prod_{i=1}^m(1+\mu_i^2)$.
\end{prop}

\section{The Barnette sphere} \label{sect:9}

We say that an abstract simplicial complex $\Si$ is the underlying simplicial complex of a (topological) toric manifold $X$ if $\Si=\Si(X)$.  An abstract simplicial complex whose geometric realization is homeomorphic to $d$-sphere $S^{d}$ is called a \emph{simplicial $d$-sphere}.  If $\Si$ is the underlying simplicial complex of a topological toric manifold of dimension $2n$, then $\Si$ is a simplicial $(n-1)$-sphere.  However, the converse does not hold in general, see Proposition A.1 in the appendix.  

The boundary complex of a simplicial $n$-polytope is a simplicial $(n-1)$-sphere.  Such a simplicial sphere is called \emph{polytopal}.  It is known that any simplicial $d$-sphere is polytopal when $d\le 2$ but there is a non-polytopal simplicial $d$-sphere when $d\ge 3$.  D. Barnette (\cite{barn71}) found a non-polytopal simplicial $3$-sphere (now known as the \emph{Barnette sphere}) with $8$ vertices.  Any simplicial $3$-sphere with less than $8$ vertices is polytopal and there are exactly two non-polytopal simplicial $3$-spheres with $8$ vertices.  The Barnette sphere is one and the other one is called the Br\"uckner sphere, see \cite[p.143]{zieg98}.  

We recall the Barnette sphere, see \cite{barn71} for the details.  A $d$-diagram is a cell complex consisting of a collection of $d$-polytopes and their faces satisfying certain conditions.  Figure \ref{fig:Barnette_sphere} is a $3$-diagram consisting of $8$ vertices $e_1,e_2,e_3,e_4,d_1,d_2,d_3,d_4$ and $18$ 3-simplices which are No.1 - No.18 in Table~\ref{table1}.  The Barnette sphere is obtained by gluing another 3-simplex along the boundary of the 3-diagram, so that it has 19 3-simplices and the glued 3-simplex is No.19 in Table~\ref{table1}.  

\begin{figure}[h]
 \begin{center}
 \includegraphics[width=4cm,clip]{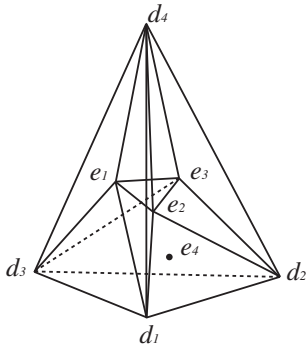}
\end{center}
\caption{3-diagram of the Barnette sphere}
\label{fig:Barnette_sphere}
\end{figure}

\begin{table}[htbp]
\begin{tabular}{|c|c|c||c|c|c||c|c|c||c|c|c|} \hline
No & simplex & sign & No & simplex & sign & No. & simplex & sign & No. & simplex & sign    \\ 
\hline \hline  1 &  $e_1 e_2 e_3 e_4$  &   $+$  & 6  &   $d_1 d_2 e_3 e_4$   &  $+$  & 11  &   $d_1 e_2 e_3 d_2$   &  $+$  &  15 &   $e_1 d_1 d_3 d_4$   &   $+$   \\
\hline          2 &  $d_1 e_2 e_3 e_4$ &   $-$  & 7  &   $e_1 d_2 d_3 e_4 $   &  $+$  & 12  &   $e_1 e_2 d_1 d_4$   &  $+$  &  16 &   $d_1 e_2 d_2 d_4$   &   $+$  \\
\hline          3 &  $e_1 d_2 e_3 e_4$ &   $-$  & 8  &   $d_1 e_2 d_3 e_4$   &  $+$  & 13  &   $e_1 d_3 e_3 d_4$   &  $+$  &  17 &   $d_3 d_2 e_3 d_4$   &   $+$   \\
\hline          4 &  $e_1 e_2 d_3 e_4$ &   $-$  & 9  &   $e_1 d_2 e_3 d_3$   &  $+$  & 14  &   $d_2 e_2 e_3 d_4$   &  $+$  &  18 &   $d_1 d_2 d_3 e_4$   &   $-$   \\
\hline          5 &  $e_1 e_2 e_3 d_4$ &   $-$  & 10 &  $e_1 e_2 d_3 d_1 $  &   $+$ &  \   &         \                       &  \ &  19 &   $d_1 d_2 d_3 d_4$  &    $+$  \\
\hline
\end{tabular}
\bigskip
\caption{3-simplices in the Barnette sphere}
\label{table1}
\end{table}

\begin{theo} \label{theo:barnette}
The Barnette sphere can be the underlying simplicial complex of a topological toric manifold but cannot be that of a toric manifold. 
\end{theo}  

\begin{proof}
The Barnette sphere $\Si_B$ is known to be the underlying simplicial complex of an ordinary complete fan over $\R$ (see \cite[\S5, Chapter III]{ewal96}).  We assign each vertex of $\Si_B$ an arbitrary non-zero edge vector of the fan on which the vertex lies.  This defines a function $b\colon \Si_B^{(1)}\to \R^4$.  We also define a function $v\colon \Si_B^{(1)}\to \Z^4$ by assigning the standard basis vectors of $\Z^4$ to the vertices $e_1, e_2, e_3,e_4$ in this order and 
\[
v(d_1)=\begin{bmatrix} 1\\
0\\
1\\
0\end{bmatrix}, \quad
v(d_2)=\begin{bmatrix} 1\\
1\\
0\\
0\end{bmatrix},\quad
v(d_3)=\begin{bmatrix} 0\\
1\\
1\\
0\end{bmatrix},\quad
v(d_4)=\begin{bmatrix} 1\\
1\\
1\\
1\end{bmatrix}.
\]   
One can easily check that this satisfies the non-singular condition in the topological fan.  Therefore, if we define $\beta=(b,v)$, then $(\Si_B,\beta)$ is a complete non-singular topological fan so that $\Si_B$ is the underlying simplicial complex of the topological toric manifold associated with $(\Si_B,\beta)$.  This proves the former part of the theorem.   
  
We shall prove the latter part of the theorem.  Suppose that $\Si_B$ is the underlying simplicial complex of a toric manifold.  Then the primitive edge vectors in the fan of the toric manifold define a function $v\colon \Si_B^{(1)}\to \Z^4$ which satisfies not only the non-singular condition
\[
\det[v(p_1),v(p_2),v(p_3),v(p_4)]=\pm 1 \quad\text{for any 3-simplex $\{p_1 ,p_2 , p_3 , p_4 \}$ in $\Si_B$}
\]
but also the condition 
\begin{equation} \label{sign}
\det[v(p_1),v(p_2),v(p_3),v(p_4)]=-\det[v(p_1),v(p_2),v(p_3),v(p_4')]
\end{equation}
for any two 3-simplices $\{ p_1 , p_2 , p_3 , p_4 \}$ and $\{p_1 , p_2 , p_3 , p_4' \}$ in $\Si_B$ sharing a 2-simplex (that is $\{ p_1 , p_2 , p_3 \}$ in this case). 

Through a modular transformation of $\Z^4$, we may assume that the function $v$ takes the standard basis vectors of $\Z^4$ at the vertices $e_1,e_2,e_3,e_4$ in this order.  Then it follows from \eqref{sign} that the signs of the determinants corresponding to the 3-simplices in $\Si_B$ are as in Table~\ref{table1}.  We write 
\[
v(d_i)= [d_{i1} , d_{i2} , d_{i3} , d_{i4}]^T \quad\text{for $i=1,2,3,4$}.
\]

\medskip
\noindent
{\bf Claim 1.}  The following equations must hold. 
\begin{alignat}{8}
d_{ii} &= -1 \ \ & \text{for} \ \ i = 1, 2, 3, 4 . \label{eq1}  \\
d_{ij} d_{ji} &= 0 \ \ &  \text{for} \ \  (i , j) = (1,2) , (2,3) , (3,1) .\label{eq2} \\
d_{j4} + d_{i4}d_{ji} &= -1 \ \ & \text{for} \ \  (i , j) = (1,2) , (2,3) , (3,1) . \label{eq3} \\
d_{ji} + d_{j4}d_{4i} &= -1 \ \ & \text{for} \ \  (i , j) = (1,2) , (2,3) , (3,1) . \label{eq4}  \\
d_{ij}-d_{kj}-d_{4j}-d_{ij}d_{k4}d_{4k} &=1 \ \ & \text{for} \ \ (i,j,k) =(1,2,3),(2,3,1),(3,1,2) . \label{eq5} \\
d_{13}d_{32}d_{21} + d_{12}d_{23}d_{31} &= 0 . \ \ \   \label{eq6}
\end{alignat}

The equations above follow by calculating the determinants corresponding to the 18 3-simplices in $\Si_B$ except the 3-simplex No.19.  We shall briefly explain how we obtain them. First, the equations (\ref{eq1}) follow from the 3-simplices No.2-5. For instance, we obtain from No.2 that 
\begin{equation*}
-1 = 
\begin{vmatrix}
d _{11} & 0 & 0 & 0 \\
d _{12} & 1 & 0 & 0 \\
d _{13} & 0 & 1 & 0 \\
d _{14} & 0 & 0 & 1 \\
\end{vmatrix} \\
=d _{11}.
\end{equation*}
Because of \eqref{eq1}, we substitute $-1$ for $d_{ii}$ in the following.  Then the equations (\ref{eq2}) follow from the 3-simplices No.6-8.  For instance, we obtain from No.6 that 
\begin{equation*}
1 =
\begin{vmatrix}
-1 & d_{21} & 0 & 0 \\
d_{12}  & -1 & 0 & 0 \\
d_{13} & d_{23} & 1 &0 \\
d_{14} & d_{24} & 0 &1 \\
\end{vmatrix}\\
=1-d_{12}d_{21},
\end{equation*}
which shows $d_{12}d_{21}=0$. Similarly, the equations (\ref{eq3}) and (\ref{eq4}) follow from the 3-simplices No.9-11 and No.12-14 respectively.  For instance, we obtain from No.9 and 12 that  
\begin{equation*}
1 =
\begin{vmatrix}
1 & d_{21} & 0 & d_{31} \\
0 & -1 & 0 & d_{32} \\
0 & d_{23} & 1 & -1 \\
0 & d_{24} & 0 & d_{34} \\
\end{vmatrix}\\
 =-d_{34} -d_{32} d_{24},\quad 
%
1 =
\begin{vmatrix}
1 & 0& -1 & d_{41} \\
0 & 1& d_{12} & d_{42} \\
0 & 0& d_{13} & d_{43} \\
0 & 0& d_{14} & -1 \\
\end{vmatrix}
= -d_{13} - d_{14}d_{43}. 
\end{equation*}
The equations (\ref{eq5}) follow from the 3-simplices No.15-17.  For instance, we obtain from No.15 that  
\begin{align*}
1 &=
\begin{vmatrix}
1&-1 & d_{31} & d_{41}  \\
0&d_{12}  & d_{32} & d_{42}  \\
0&d_{13} & -1 & d_{43}  \\
0&d_{14} & d_{34} & -1  \\
\end{vmatrix}
=
d_{12} + d_{14}d_{43}d_{32} + d_{34}d_{42}d_{13} + d_{14}d_{42} - d_{12}d_{34}d_{43} + d_{13}d_{32} \\
&=
d_{12} + ( -1 -d_{13} ) d_{32} + d_{34}d_{42}d_{13} 
+ (-1 -d_{34}d_{13})d_{42} - d_{12}d_{34}d_{43} + d_{13}d_{32} \\
&=
d_{12}-d_{32}-d_{42}-d_{12}d_{34}d_{43} 
\end{align*}
where we used (\ref{eq3}) and (\ref{eq4}) for $(i,j) = (3,1)$ at the third equality above.  Finally the equation (\ref{eq6}) follows from the 3-simplex No.18: 
\begin{align*}
-1 &=
\begin{vmatrix}
-1 & d_{21} & d_{31} & 0 \\
d_{12}  & -1 & d_{32} & 0 \\
d_{13} & d_{23} & -1 &0 \\
d_{14} & d_{24} & d_{34} &1 \\
\end{vmatrix}
= -1 + d_{13}d_{32}d_{21} + d_{12}d_{23}d_{31} + d_{13}d_{31} +d_{23}d_{32} +d_{12}d_{21}\\
&= -1 + d_{13}d_{32}d_{21} + d_{12}d_{23}d_{31}
\end{align*}
where we used (\ref{eq2}) at the last equality above.  This proves Claim 1.  

The desired fact that the Barnette sphere $\Si_B$ cannot be the underlying simplicial complex of a toric manifold follows once we prove the following claim.  

\medskip
\noindent
{\bf Claim 2.}  The equations (\ref{eq2}) - (\ref{eq6}) have no integer solution. 

\medskip

First we note that the equations (\ref{eq2}) and (\ref{eq6}) imply that 
\begin{equation} \label{two equations}
\text{$d_{13}d_{32}d_{21}=0$\quad and\quad $d_{12}d_{23}d_{31}=0$.}
\end{equation}
Therefore 
\[
D:=\sharp \{ d_{ji} \mid d_{ji}=0 \ \text{for} \ (i,j)\in \{(1,2),(2,3),(3,1)\} \} \text{ is either 1, 2 or 3}. 
\]
We shall observe that this does not occur, which implies the claim. 

(1) The case where $D=3$.  In this case, it follows from  (\ref{eq3}) and (\ref{eq4}) that $d_{j4}=-1$ and $d_{4i} = 1$ for $1\le i,j\le 3$.  
Substituting these in (\ref{eq5}) and using $d_{ji}=0$ for  $(i , j) = (1,2) , (2,3) , (3,1)$, we get $d_{ij} = 1$ for  $(i , j) = (1,2) , (2,3) , (3,1)$.  
However this contradicts the fact $d_{12} d_{23}d_{31} = 0$ in \eqref{two equations}.  
  
(2) The case where $D=1 $ or $2$.  In this case $d_{ji}\not=0$ and $d_{j'i'}=0$ for some $(i,j)$ and $(i',j')$ in $\{(1,2),(2,3),(3,1)\}$.  Because of symmetry of the suffixes, we may assume that $d_{32}\neq 0$ and $d_{13}=0$ without loss of generality.  Then $d_{23}=0$ from (\ref{eq2}) with $(i,j)=(2,3)$ and $d_{14}=-1$ from (\ref{eq3}) with $(i,j)=(3,1)$, and then $d_{43}=1$ follows from (\ref{eq4}) with $(i,j)=(3,1)$.  However these values do not satisfy (\ref{eq5}) with $(i,j,k)=(2,3,1)$.  This completes the proof of the claim and hence the theorem. 
\end{proof}

The following corollary gives evidence that topological toric manifolds are much more abundant than toric manifolds.  The reader will find another evidence at the end of Section~\ref{sect:11}. 

\begin{coro} \label{coro:abun}
If $n\ge 4$, then there are infinitely many simplicial $(n-1)$-spheres each of which can be the underlying simplicial complex of a topological toric manifold but cannot be that of a toric manifold. 
\end{coro}

\begin{proof}
We do not use the 3-simplex No.19 in the proof of Theorem~\ref{theo:barnette}.  Therefore the proof also shows that any simplicial complex obtained by performing stellar subdivisions on the 3-simplex finitely many times cannot be the underlying simplicial complex of a toric manifold.  On the other hand, any such simplicial complex can be the underlying simplicial complex of a topological toric manifold by Lemma A.2 in the appendix.  This together with Lemma A.3 in the appendix implies the corollary. 
\end{proof}

\section{Relation with quasitoric manifolds} \label{sect:10}

Two topological analogues of a toric manifold are known so far.  One is what is now called a quasitoric manifold introduced by Davis-Januszkiewicz \cite{da-ja91} and the other is a torus manifold introduced by Masuda \cite{masu99} and Hattori-Masuda \cite{ha-ma03}.  We will discuss the relation of topological toric manifolds with quasitoric manifolds in this section and with torus manifolds in the next section.  

The orbit space of a toric manifold $X$ by the restricted action of $(S^1)^n$ is often a simple polytope.  In fact, this is the case when $X$ is projective (since $X/(S^1)^n$ can be identified with the image of a moment map) or when $\dim_\C X\le 3$ (this is trivial when $\dim_\C X=1,2$ and follows from a well-known theorem of Steinitz (see \cite{zieg98}) when $\dim_\C X=3$).  
Davis-Januszkiewicz assert in \cite[line 15 in p.419]{da-ja91} that the orbit space of a toric manifold by the action of $(S^1)^n$ can always be identified with a simple polytope.  But this is uncertain although no counterexamples to their assertion are known.\footnote{Civan \cite{civa03} claims that 
there exists a toric manifold of complex dimension $4$ whose orbit space by the restricted action of $(S^1)^4$ is combinatorially different from any simple 4-polytope, but his proof is unclear unfortunately.}  Anyway, modeled by toric manifolds in algebraic geometry, Davis-Januszkiewicz introduced the following topological analogue of a toric manifold.  
 
\begin{defi}[Davis-Januszkiewicz \cite{da-ja91}] 
A closed smooth manifold $M$ of dimension $2n$ with a smooth action of $(S^1)^n$ is called a \emph{quasitoric manifold}\footnote{In \cite{da-ja91} this  is called a \emph{toric manifold} but since this terminology was already used in algebraic geometry, Buchstaber-Panov \cite{bu-pa02} started using the word \emph{quasitoric manifold}.} over a simple polytope $Q$ if 
\begin{enumerate}
\item the action of $(S^1)^n$ on $M$ is \emph{locally standard}, i.e. any point of $M$ has an invariant open neighborhood equivariantly diffeomorphic to an open invariant subset of a direct sum of complex one-dimensional representation spaces of $(S^1)^n$,  
\item the orbit space $M/(S^1)^n$ is the simple polytope $Q$.
\end{enumerate}
\end{defi} 

\begin{rema}
The restricted action of $(S^1)^n$ on a topological toric manifold $X$ is locally standard so that condition (1) above is satisfied.  The orbit space $X/(S^1)^n$ is a manifold with corners whose faces are all contractible and any intersection of faces is connected unless it is empty as remarked before.  So it is a simple polytope when $n\le 3$, so condition (2) is satisfied for any topological toric manifold $X$ when $n\le 3$.  However, since suspension of a non-polytopal simplicial sphere is still non-polytopal (in fact, the link of each vertex in a polytopal sphere is also polytopal; but the link of a new vertex in the suspension of a non-polytopal simplicial sphere is the non-polytopal simplicial sphere), the examples in Section~\ref{sect:9} show that there are topological toric manifolds which do not satisfy condition (2) when $n\ge 4$.  
\end{rema}

Let $M$ be a quasitoric manifold of dimension $2n$ with a simple polytope $Q$ as the orbit space.  Then there are only finitely many codimension two closed connected submanifolds fixed pointwise under some $S^1$-subgroup of $(S^1)^n$.  These submanifolds are also called \emph{characteristic submanifolds} of $M$ and denoted by $M_1,\dots,M_m$.  As for a topological toric manifold $X$, the characteristic submanifolds of $X$ are defined as codimension two closed connected submanifolds of $X$ fixed pointwise under some $\C^*$-subgroup of $(\C^*)^n$ but they may also be defined as those submanifolds fixed pointwise under some $S^1$-subgroup of $(S^1)^n$.    

As is well-known there is a natural isomorphism 
\[
\lambda\colon \Z^n\to \Hom(S^1,(S^1)^n)
\]
where $\Hom(S^1,(S^1)^n)$ is the group of smooth homomorphisms from $S^1$ to $(S^1)^n$.  In fact, the $\lambda$ above maps an element $q=(q_1,\dots,q_n)\in \Z^n$ to a homomorphism $\lambda_q\colon S^1\to (S^1)^n$ sending $g\in S^1$ to $(g^{q_1},\dots,g^{q_n})\in (S^1)^n$.  Let $v_i(M)$ be a primitive vector in $\Z^n$ such that $\lambda_{v_i(M)}$ is the $S^1$-subgroup of $(S^1)^n$ which fixes $M_i$ pointwise.  Note that $v_i(M)$ is defined only up to sign so that it defines a unique element in the quotient $\Z^n/\{\pm 1\}$.  In order to avoid the ambiguity of the sign, 
we need to specify an omniorientation on $M$ as is done for the topological toric manifolds.  

Similarly to the topological toric case, 
\[
\Si(M):=\{ I\subset [m]\mid M_I:=\bigcap_{i\in I}M_i\not=\emptyset\}
\]
is an abstract simplicial complex of dimension $n-1$ and the smoothness of $M$ implies the following non-singular condition:
\begin{equation} \label{smooth}
\text{$\{v_i(M)\}_{i\in I}$ is a part of a basis of $\Z^n$ whenever  $I\in \Si(M)$.}
\end{equation}  

Through the quotient map $M\to Q=M/(S^1)^n$, the characteristic submanifolds $M_1,\dots,M_m$ of $M$ bijectively correspond to the facets $Q_1,\dots,Q_m$ of $Q$. Note that for $I\subset [m]$, 
\begin{equation} \label{SiP}
\text{$I\in\Si(M)$\quad if and only if \quad $\bigcap_{i\in I}Q_i\not=\emptyset$}, 
\end{equation}
which means that the face poset of $Q$ coincides with the inverse poset of $\Si(M)$, in other words, the boundary complex of the simplicial polytope $P$ dual to $Q$ agrees with $\Si(M)$.  Therefore the orbit space $Q$ contains all the information of $\Si(M)$.  We may regard the collection of the $v_i(M)$'s as a map 
\begin{equation} \label{vM}
v(M)\colon [m]=\Si^{(1)}(M)\to \Z^n/\{\pm 1\}
\end{equation}
sending $i\in [m]$ to $v_i(M)\in\Z^n/\{\pm 1\}$. 

Conversely, if there is a map 
\[
v\colon [m]=\Si^{(1)}(M)\to \Z^n/\{\pm 1\}
\]
which satisfies condition \eqref{smooth}, then one can construct a quasitoric manifold $M(Q,v)$ over $Q$ depending only on $Q$ and $v$ (see \cite[Section 1.5]{da-ja91}) and the following is known to hold. 

\begin{prop}[see Proposition 1.8 in \cite{da-ja91}] \label{prop:daja}
Let $M$ be a quasitoric manifold over a simple polytope $Q$ and let $v(M)$ be the map associated with $M$.  Then, there is an equivariant homeomorphism $M\to M(Q,v(M))$ covering the identity on $Q$.  In particular, the pair $(Q,v(M))$ determines the equivariant homeomorphism type of $M$. 
\end{prop}

Here is a relation between topological toric manifolds and quasitoric manifolds.  

\begin{theo} \label{theo:quasi}
The family of topological toric manifolds as $(S^1)^n$-manifolds contains the family of quasitoric manifolds up to equivariant homeomorphism.  To be more precise, the former family agrees with the latter when $n\le 3$ and properly contains the latter when $n\ge 4$ up to equivariant homeomorphism.  
\end{theo}

\begin{proof}
Let $M$ be a quasitoric manifold over a simple polytope $Q$. We put the simplicial polytope $P$ dual to $Q$ in $\R^n$ in such a way that the origin of $\R^n$ lies in the interior of $P$ and form cones by joining the origin and proper faces of $P$.  This produces an ordinary complete simplicial fan.  We identify the boundary complex of $P$ with $\Si(M)$ and take any non-zero vector, denoted $b_i$, lying on the 1-dimensional cone corresponding to the vertex $i\in\Si^{(1)}(M)$.  We take $c_i$ to be an arbitrary vector of $\R^n$, view $v_i(M)$ as an element of $\Z^n$ by taking a lift and set $\beta_i(M)=(b_i+\sqrt{-1}c_i,v_i(M))$.  Then $\De(M):=(\Si(M),\beta)$ is a complete, non-singular topological fan, where $\beta$ is the map $\Si^{(1)}(M)\to \RZRn=\MR^n$ sending $i$ to $\beta_i(M)$.  The topological toric manifold $X(\De(M))$ associated with $\De(M)$ is equivariantly homeomorphic to $M$ as $(S^1)^n$-manifolds by Proposition~\ref{prop:daja}.  (Note. Since there are uncountably many choices of the $b_i+\sqrt{-1}c_i$, there are uncountably many such topological toric manifolds $X(\De(M))$ by Theorem~\ref{theo:class}.)  This proves the former statement in the theorem. 

The latter statement in the theorem follows from the Remark after the definition of a quasitoric manifold above.  
\end{proof}
\begin{exam}
	Let $M$ be a closed simply-connected smooth manifold of dimension $4$ on which $(S^1)^2$ acts effectively. It follows from the classification result on such manifolds (see \cite{or-ra70}) that $M$ is diffeomorphic to $S^4$,  a connected sum of $m$ copies of $S^2\times S^2$, or a connected sum of $k$ copies of $\C P^2$ and $\ell $ copies of $\overline{\C P^2}$, where $\overline{\C P^2}$ is $\C P^2$ with orientation reversed. Here any values of $m$ , $k$ and $\ell$ can occur while if $M$ is a toric manifold, then $M$ is diffeomorphic to $S^2\times S^2$, $\C P^2$ or their connected sum with a finitely many copies of $\overline{\C P^2}$.    
The action of $(S^1)^2$ on $M$ is locally standard and the orbit space $M/(S^1)^2$ is a $q$-gon where $q=2$ if and only if $M=S^4$. Therefore, unless $M=S^4$, $M$ with the action of $(S^1)^2$ is a quasitoric manifold. Theorem \ref{theo:quasi} tells us that the action of $(S^1)^2$ on $M$ extends to an action of $(\C ^*)^2$ which locally looks like a smooth faithful representation of $(\C ^*)^2$ and Corollary~\ref{coro:class2} implies that there are uncountably many extended actions of $(\C ^*)^2$ on $M$ up to equivariant diffeomorphism. 
\end{exam}
\begin{rema}
It is natural to expect that the \lq\lq homeomorphic" in Theorem~\ref{theo:quasi} can be replaced by \lq\lq diffeomorphic", but for that we would need to establish the diffeomorphism version of Proposition~\ref{prop:daja} that the pair $(Q,v(M))$ determines the equivariant diffeomorphism type of $M$.
\end{rema}

\section{Relation with torus manifolds} \label{sect:11}

A \emph{torus manifold} is a closed connected orientable smooth manifold $M$ of dimension $2n$ with an effective smooth action of $(S^1)^n$ having a fixed point.\footnote{In \cite{ha-ma03} an omniorientation is incorporated in the definition of a torus manifold.}  Obviously a topological toric manifold as an $(S^1)^n$-manifold is a torus manifold and the family of torus manifolds is much larger than the family of topological toric manifolds as $(S^1)^n$-manifolds as is seen below. 

\begin{exam}
We regard $2n$-sphere $S^{2n}$ as the unit sphere of $\C^n\oplus\R$.  Then  $S^{2n}$ has a natural action of $(S^1)^n$ defined by 
\[
(z_1,\dots,z_n,y)\to (g_1z_1,\dots,g_nz_n,y) 
\]
where $(z_1,\dots,z_n)\in \C^n, y\in \R$ and $(g_1,\dots,g_n)\in (S^1)^n$.  This action has two fixed points, the north pole and the south pole, so $S^{2n}$ with the action of $(S^1)^n$ is a torus manifold.  However, $S^{2n}$ for $n\ge 2$ cannot be the underlying manifold of a topological toric manifold by Proposition~\ref{prop:cohom} because $H^*(S^{2n})$ for $n\ge 2$ is not generated by degree two elements as a ring.   The orbit space $S^{2n}/(S^1)^n$ is an $n$-dimensional manifold with corners, see figure~\ref{figure3}.  

\begin{center}
\begin{figure}[h]
	\includegraphics{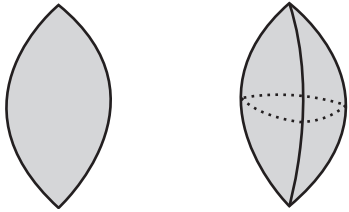}
\caption{$S^4/(S^1)^2$ and $S^6/(S^1)^3$}
\label{figure3} 
\end{figure}
\end{center}
\end{exam}

The action of $(S^1)^n$ on a torus manifold $M$ is not necessarily locally standard and hence the orbit space $M/(S^1)^n$ may not be a manifold with corners, and even if the action is locally standard (so that $M/(S^1)^n$ is a manifold with corners), $M/(S^1)^n$ and its proper faces may not be contractible or may have non-trivial cohomology.  Here are simple examples. 

\begin{exam}
Torus manifolds of dimension $4$ whose actions are \emph{not} locally standard can be constructed as follows (see also \cite{or-ra70}). 
	Let $D^2$ denote the unit disc in $\C$. Let $m$ be an integer greater than $1$. Consider the action of $(S^1)^2$ on $(S^1)^2\times D^2$ given by
	\begin{equation}\label{eq:Seifert}
		(h_1,h_2,v) \mapsto (g_1h_1,g_2^mh_2, g_2^{-1}v)
	\end{equation}
	for $(g_1,g_2) \in (S^1)^2$. 
	The isotropy subgroup at $(h_1,h_2,v) \in (S^1)^2 \times D^2$ is trivial if $v \neq 0$  and $\{ 1\} \times \Z /m$ if $v= 0$ . Consider the  diffeomorphism
	\begin{equation*}
	\begin{split}
		\varphi : \partial ((S^1)^2 \times D^2)=(S^1)^3 &\to (S^1)^3\\
		(h_1,h_2,v)&\mapsto (h_1,h_2v^{m-1},h_2v^m).
	\end{split}
	\end{equation*}
	Then, we have
	\begin{equation*}
		\varphi ((g_1,g_2)\cdot (h_1,h_2,v))=(g_1h_1,g_2h_2v^{m-1},h_2^mv).
	\end{equation*}
\indent Take any torus manifold $M$ of dimension $4$ and let $N$ be a closed equivariant tubular neighborhood of a free orbit of $M$. Thanks to the slice theorem, $N$ is equivariantly diffeomorphic to $(S^1)^2 \times D^2$ for the action defined by
	\begin{equation}\label{eq:tubular}
		(g_1,g_2)\cdot (h_1',h_2',v)=(g_1h_1',g_2h_2',v).
	\end{equation}
\indent Let $\psi $ denote an equivariant diffeomorphism from $N$ to $(S^1)^2 \times D^2$ with the action given by \eqref{eq:tubular}. Then, the map $\varphi ^{-1}\circ \psi : \partial N \to \partial ((S^1)^2 \times D^2)$ is an equivariant diffeomorphism for the action given by \eqref{eq:Seifert}. We glue $M \setminus (\Int N)$ and $(S^1)^2\times D^2$ by $\varphi ^{-1}\circ \psi$ along their boundaries, where $\Int N$ denotes the interior of $N$. The resulting manifold 
	\begin{equation*}
		M'=M\setminus (\Int N)\cup _{\varphi ^{-1}\circ \psi}(S^1)^2 \times D^2
	\end{equation*}
	is a torus manifold of dimension $4$
	having a point whose isotropy subgroup is nontrivial and finite. Namely, $(S^1)^2$-action on $M'$ is not locally standard. 
\end{exam}

\begin{exam}
Consider the standard $(S^1)^2$-action on $\C P^2$ defined by 
\[
[z_0,z_1,z_2]\to [z_0,g_1z_1,g_2z_2]
\]
where $[z_0,z_1,z_2]\in \C P^2$ and $(g_1,g_2)\in (S^1)^2$.  This action is locally standard and the orbit space $\C P^2/(S^1)^2$ can be identified with a triangle $\Delta$.  Let $q\colon \C P^2\to \C P^2/(S^1)^2=\Delta$ be the quotient map.  A free $(S^1)^2$-orbit in $\C P^2$ corresponds to an interior point in the triangle $\Delta$ through the map $q$.  We take a closed disk $D_1$ in the interior of $\Delta$.  Note that $q^{-1}(D_1)$ can be identified with $D^2\times (S^1)^2$ as $(S^1)^2$-manifolds where $(S^1)^2$ acts trivially on $D^2$ and as group multiplication on $(S^1)^2$.  We remove $q^{-1}(\Int D_1)$ from $\C P^2$, where $\Int D_1$ denotes the interior of $D_1$.  On the other hand, we take a closed disk $D_2$ in a closed orientable surface $\Sigma_g$ with genus $g\ge 1$ and glue $(\Sigma_g\backslash \Int D_2)\times (S^1)^2$ to $\C P^2\backslash q^{-1}(\Int D_1)$ equivariantly along their boundary.  This produces a torus manifold of dimension $4$ whose orbit space is the connected sum of $\Delta$ with $\Sigma_g$ at interior points.  Since $g\ge 1$, the orbit space is non-acyclic, see figure~\ref{figure4}.  

\begin{figure}[h]
\begin{center}
	\includegraphics{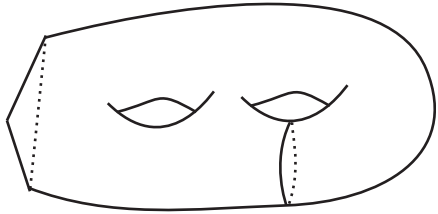}
\end{center}
\caption{non-acyclic orbit space with $g=2$}
\label{figure4}
\end{figure}

\end{exam}   
  
The following theorem from \cite{ma-pa06} clarifies the relation between the topology of a torus manifold $M$ and the topology and combinatorics of the orbit space $M/(S^1)^n$.  

\begin{theo}[\cite{ma-pa06}]\label{theo:ma-pa}
Let $M$ be a torus manifold of dimension $2n$.  
\begin{enumerate}
\item $H^{odd}(M)=0$ if and only if the action of $(S^1)^n$ on $M$ is locally standard and $M/(S^1)^n$ is \emph{face acyclic}, i.e. 
every face of $M/(S^1)^n$ (even $M/(S^1)^n$ itself) is acyclic.
\item $H^*(M)$ is generated by degree two elements as a ring if and only if the action of $(S^1)^n$ on $M$ is locally standard and $M/(S^1)^n$ is a homology polytope, i.e. in addition to the face acyclicity, any intersection of faces in $M/(S^1)^n$ is connected unless it is empty.
\end{enumerate}
\end{theo}  
  
Although the family of torus manifolds is much larger than the family of topological toric manifolds as $(S^1)^n$-manifolds, one can associate a combinatorial object called a multi-fan to a torus manifold (\cite{ha-ma03}, \cite{masu99}), which is another generalization of an ordinary fan.  We shall review it to compare with our topological fan.  

Let $M$ be a torus manifold of dimension $2n$.  As before, a characteristic submanifold of $M$ is a codimension two closed connected submanifold fixed pointwise under some $S^1$-subgroup of $(S^1)^n$.  There are only finitely many characteristic submanifolds in $M$ which we denote by $M_1,\dots,M_m$.  Then the simplicial complex $\Si(M)$ can be defined as before.  An omniorientation on $M$ is a choice of an orientation on $M$ and on each $M_i$.  Suppose $M$ is omnioriented.  Then the vectors $\v_i(M)$ for $i=1,\dots,m$ can be defined without ambiguity of sign similarly to the topological toric case, so that we have a map
\[
v(M)\colon [m]=\Sigma^{(1)}(M)\to \Z^n
\]
sending $i$ to $v_i(M)$. 

Let $I\in \Si^{(n)}(M)$.  Then $M_I:=\bigcap_{i\in I}M_i$ is not necessarily one point but consists of finitely many fixed points.  For each fixed point $p\in M_I$ we have an identity 
\begin{equation} \label{TpM}
\tau_pM=\bigoplus_{i\in I}\nu_i|_p\quad\text{as real $(S^1)^n$-representation spaces}
\end{equation}
similarly to the topological toric case \eqref{TpX}.  Since $M$ is omnioriented, both sides above have orientations as discussed in Section~\ref{sect:3}.  We set $w(p)=1$ if the orientations agree and $w(p)=-1$ if they disagree, and define 
\[
\begin{split}
&w^+(M)(I):=\#\{p\in M_I\mid w(p)=1\},\quad w^-(M)(I):=\#\{p\in M_I\mid w(p)=-1\},\\
&w(M)(I):=w^+(M)(I)-w^-(M)(I) \quad\text{for $I\in \Si^{(n)}(M)$}.
\end{split}
\]
When $M$ is a topological toric manifold, $M_I$ is one point for $I\in \Si^{(n)}(M)$ so that $w(M)(I)=\pm 1$.  If $M$ is a toric manifold, then $M$ has a canonical omniorientation induced from the complex structures on $M$ and $M_i$'s, and $w(M)(I)=1$ with the canonical omniorientation.  

\begin{defi}
A triple $(\Si(M),v(M), w^{\pm}(M))$ is called the \emph{multi-fan} associated to $M$.  
\end{defi}

\begin{rema}
Let $X$ be an omnioriented topological toric manifold of dimension $2n$ and let $\De(X)=(\Si(X),\beta(M))$ be the topological fan associated with $X$.  Since $X$ is omnioriented, one can define $w^\pm(X)$ in the same way as $w^\pm(M)$ above.  Then the multi-fan associated with $X$ as an $(S^1)^n$-manifold is the triple $(\Si(X),v(M),w^\pm(X))$.  Note that one can find $w^\pm(X)$ from $\beta(X)$.  In fact, it follows from Lemma~\ref{lemm:orient} that  if the determinant of the matrix formed from $\{\beta_i(X)\}_{i\in I}$ (viewed as a matrix of size $2n$) is positive (resp. negative), then $(w^+(X)(I),w^-(X)(I))=(1,0)$ (resp. $(0,1)$) and hence $w(X)(I)=1$ (resp. $-1$) for $I\in \Si^{(n)}(X)$.    
\end{rema}

Unlike the topological toric case, the multi-fan $(\Si(M),v(M), w^{\pm}(M))$ does not determine the torus manifold $M$.  However, it contains a lot of geometrical information on $M$ because it contains the complete information on the tangential representation at each fixed point.  For instance, important characteristic numbers of $M$ such as signature, Hirzebruch $T_y$-genus (or $\chi_y$-genus) and elliptic genus can be described in terms of the multi-fan of $M$, see \cite{ha-ma03}, \cite{ha-ma05}, (a similar description can be found in \cite{pano99} and \cite{pano01}).  Therefore, it is possible to describe those invariants of a topological toric manifold $X$ in terms of the topological fan $\De(X)$ because the topological fan $\De(X)$ determines the multi-fan of $X$ as remarked above.  

For each $I\in\Si(M)$, we form a cone $\angle v_I(M)$ in $\R^n$ generated by the vectors $v_i(M)$ for $i\in I$.  When $M$ is a toric manifold, these cones do not overlap and form an ordinary complete non-singular fan associated with $M$.  Unless $M$ is a toric manifold, the cones $\angle v_I(M)$ may overlap.  However, they are not placed at random as is shown in the theorem below.  In \cite{ha-ma03}, the Todd genus is defined for an omnioriented torus manifold $M$ in such a way that when the omniorientation is induced from an $(S^1)^n$-invariant unitary (or weakly almost complex) structure, then our Todd genus agrees with the Todd genus of $M$ with the unitary structure.  

\begin{theo}[Theorem 4.2 in \cite{masu99}]
Let $M$ be an omnioriented torus manifold of dimension $2n$ and let $v$ be a generic element in $\R^n$.  Then the Todd genus $T[M]$ of $M$ is given by 
\[
T[M]=\sum_{I \in \Si ^{(n)} : v\in \angle v_I(M)}w(M)(I).
\]
In particular, the right hand side above is independent of the choice of the generic element $v$ because so is the left hand side.  
\end{theo} 

As is well-known, the Todd genus of a toric manifold is one.  This fact can also be seen from the theorem above because when $M$ is a toric manifold, $w(M)(I)=1$ for any $I\in \Si^{(n)}(M)$ and the cones $\angle v_I(M)$ do not overlap as remarked before.  When $M$ is a topological toric manifold, $w(M)(I)=\pm 1$ for any $I\in\Si^{(n)}(M)$ but the Todd genus $T[M]$ can assume any integer value.  In fact, such examples are constructed for torus manifolds of dimension 4 in \cite[Section 4]{masu99} but they are actually quasitoric manifolds and since any quasitoric manifold is a topological toric manifold by Theorem~\ref{theo:quasi}, the Todd genus can assume any integer value for topological toric manifolds.  This gives another evidence that topological toric manifolds are much more abundant than toric manifolds, cf. Corollary~\ref{coro:abun}.

\section{Real topological toric manifolds and small covers} \label{sect:12}

One can see that the argument developed in previous sections work\textcolor{blue}{s} with $\C, S^1, \Z$ replaced by $\R, S^0, \Z/2=\{0,1\}$ respectively with a little modification. In fact, the situation becomes simpler because any smooth group endomorphism of $\R^*=\R_{>0}\times S^0$ is of the form 
\[
g\to |g|^b(\frac{g}{|g|})^{\bar v} \quad \text{with $(b,\bar v)\in \R\times \Z/2$}.
\]
In this section we will briefly discuss a real analogue of a topological toric manifold and a quasitoric manifold.  

\begin{defi}
We say that a closed smooth manifold $Y$ of dimension $n$ with an effective \emph{smooth} action of $(\R^*)^n$ having an open dense orbit is a (compact) \emph{real topological toric manifold} if it is covered by finitely many invariant open subsets each of which is equivariantly diffeomorphic to a direct sum of real one-dimensional \emph{smooth} representation spaces of $(\R^*)^n$.  
\end{defi}

Similarly to the topological toric case, one can define characteristic submanifolds $Y_1,\dots,Y_m$ of $Y$, each of which is a codimension one connected closed submanifold fixed pointwise under some $\R^*$-subgroup of $(\R^*)^n$, and associate a combinatorial object $\De(Y)=(\Si(Y),\bar\beta(Y))$.  Here 
 \[
\Si(Y):=\{ I\subset [m]\mid Y_I:=\bigcap_{i\in I}Y_i\not=\emptyset\}\cup \{\emptyset\}
\]
and 
\[
\bar\beta(Y)\colon [m]=\Si^{(1)}(Y)\to (\Z/2)^n
\]
which sends $i\in[m]$ to the $\R^*$-subgroup of $(\R^*)^n$ fixing $Y_i$ through a natural isomorphism 
\[
\Hom(\R^*,(\R^*)^n)\cong \R^n\times (\Z/2)^n.
\]

\begin{defi}
Let $\Si$ be an abstract simplicial complex of dimension $n-1$ (with the empty set $\emptyset$ added) and let 
\[
\bar\beta\colon \Si^{(1)} \to \R^n\times(\Z/2)^n.
\]
We express $\bar\beta(i)=\bar\beta_i=(b_i,\bar v_i)\in \R^n\times(\Z/2)^n$.  Then a pair $\De=(\Si,\bar\beta)$ is called a \emph{(simplicial) real topological fan} of dimension $n$ if the following are satisfied. 
\begin{enumerate}
\item $b_i$'s for $i\in I$ are linearly independent whenever $I\in \Si$, and $\angle \b_I\cap \angle \b_J=\angle \b_{I\cap J}$ for any $I,J\in \Si$.
(In short, the collection of cones $\angle b_I$ for all $I\in \Si$ is an ordinary simplicial fan although the $\b_i$'s are not necessarily in $\Z^n$.) 
\item $\bar v_i$'s for $i\in I$ are linearly independent over $\Z/2$ whenever $I\in \Si$.  
\end{enumerate} 
We say that a topological fan $\De$ of dimension $n$ is \emph{complete} if $\bigcup_{I\in \Si}\angle \b_I=\R^n$ and \emph{non-singular} if $\bar\v_i$'s for $i\in I$ form a part of a basis of $(\Z/2)^n$ whenever $I\in\Si$.  
\end{defi}

The argument in Sections~\ref{sect:5} and~\ref{sect:6} works with $\C$ replaced by $\R$ and that in Section~\ref{sect:7} works with $S^1$ replaced by $S^0$.  Therefore the results in Section~\ref{sect:8} hold for real topological toric manifolds with suitable modification. One big difference between topological toric manifold and real topological toric manifolds is that a topological toric manifold is simply connected (Proposition~\ref{prop:simply}) while a real topological toric manifold is not simply connected.  

The real analogue of a quasitoric manifold is the following.   

\begin{defi}[Davis-Januszkiewicz \cite{da-ja91}] 
A closed smooth manifold $N$ of dimension $n$ with a smooth action of $(S^0)^n$ is called a \emph{small cover} over a simple polytope $Q$ if 
\begin{enumerate}
\item the action of $(S^0)^n$ on $N$ is locally standard, and 
\item the orbit space $N/(S^0)^n$ is the simple polytope $Q$.
\end{enumerate}
\end{defi} 

As is developed in \cite{da-ja91}, almost all arguments for quasitoric manifold work for small covers with $S^1$ and $\Z$ replaced by $S^0$ and $\Z/2$ respectively, and Proposition~\ref{prop:daja} holds for small covers with suitable modification.  Thus the analogous theorem to Theorem~\ref{theo:quasi} holds for real topological toric manifolds and small covers.  

A quasitoric manifold $M$ of dimension $2n$ admits an involution called \emph{conjugation} (see \cite{da-ja91}) and its fixed point set is a small cover.  However, it is not always the case that a topological toric manifold admits such a conjugation.  We say that a topological fan $\De=(\Si,\beta)$ where $\beta_i=(b_i+\sqrt{-1}c_i,v_i)$ is \emph{involutive} if $c_i=0$ for all $i\in\Si^{(1)}$.  Accordingly, we say that a topological toric manifold $X$ is \emph{involutive} if the topological fan associated with $X$ is involutive.  Note that any toric manifold is involutive.  Suppose that our topological toric manifold $X$ is involutive.  Since we may assume $X=X(\De)$ by Theorem~\ref{theo:class}, $X$ has an involution (called the conjugation on $X$) induced from the complex conjugation on $\C^n$ by Lemma~\ref{lemm:conjugate}.  The fixed point set of the conjugation on $X$, denoted $\RX$, is a real topological toric manifold.  Clearly $\De(\RX)=(\Si(X),\bar\beta(X))$ where $\bar\beta(X)=(b(X),\bar v(X))\in \R^n\times(\Z/2)^n$ and $\bar v(X)$ denotes the mod 2 reduction of $v(X)$.  The authors do not know whether any real topological toric manifold can be realized as the fixed point set of the conjugation in some involutive topological toric manifold.

\section*{Appendix}

In this appendix, we collect a few results on underlying simplicial complexes of (topological) toric manifolds.    

\medskip
\noindent
{\bf Proposition A.1.} 
\emph{If $n\le 3$, then any simplicial $(n-1)$-sphere can be the underlying simplicial complex of a topological toric manifold.  However, if $n\ge 4$, then there are infinitely many simplicial $(n-1)$-spheres which cannot be the underlying simplicial complex of a topological toric manifold.}

\begin{proof}
The proof is essentially same as in \cite{da-ja91}.  
The proposition is clear when $n\le 2$.  Let $\Si$ be any simplicial $2$-sphere.  It is isomorphic to the boundary complex of some simplicial $3$-polytope $P$ by a theorem of Steinitz.  We put $P$ in $\R^3$ in such a way that $P$ contains the origin of $\R^3$ in its interior.  Let $x_1,\dots,x_m$ be the vertices of $P$.  We take $b_i$ in the definition of a topological fan to be the position vector of $x_i$ in $\R^3$.  On the other hand, the four color theorem ensures that one can assign one of four vectors $e_1,e_2,e_3,e_1+e_2+e_3$, where $e_1,e_2,e_3$ denote the standard basis of $\R^3$, to each vertex of $P$ in such a way that different vectors are assigned to any two vertices joined by an edge in $P$. We take $v_i$ to be the integral vector assigned to the vertex $x_i$.  Then a pair $(\Si,\beta)$ with $\beta_i=(b_i,v_i)$ defines a complete non-singular topological fan of dimension $3$, so that the $\Si$ is the underlying simplicial complex of a topological toric manifold by Theorem~\ref{theo:class}. 

The latter statement in the proposition follows from \cite[Nonexample 1.22]{da-ja91}.  We shall reproduce the argument for the reader's convenience. For any integers $m> n(\ge 2)$, there is an $n$-dimensional simplicial polytope with $m$ vertices denoted by $C^n(m)$ and called a cyclic polytope.  Then the boundary complex of $C^n(m)$ cannot be the underlying simplicial complex of a topological toric manifold if $n\ge 4$ and $m\ge 2^n$.  The reason is as follows. Since $n\ge 4$, the 1-skeleton of $C^n(m)$ is known to be a complete graph (i.e. any two vertices in $C^n(m)$ are joined by an edge), and since $m\ge 2^n$, any nonzero integral vectors assigned to the $m$ vertices must have the same mod 2 reduction at some two vertices; so it is impossible to assign integral vectors $v_i$ to the $m$ vertices in such a way that they satisfy the non-singular condition over $\Z$ in the definition of a non-singular topological fan, see Section~\ref{sect:3}.  Therefore, the boundary complex of $C^n(m)$ cannot be the underlying simplicial complex of a topological toric manifold if $n\ge 4$ and $m\ge 2^{n}$, proving the latter statement in the proposition.  
\end{proof}

Let $\sigma$ be a simplex of maximal dimension in $\Si$.  We remove $\sigma$ from $\Si$ and add a new vertex, say $x$, and all cones in the join $x*\partial \sigma$ to $\Si$, where $\partial \sigma$ denotes the boundary complex of $\sigma$.  This operation produces a new simplicial complex and is called a \emph{stellar subdivision}.  

\medskip
\noindent
{\bf Lemma A.2.}
\emph{If a simplicial complex $\Si$ is the underlying simplicial complex of a topological toric manifold, then any simplicial complex obtained from $\Si$ by performing a stellar subdivision on $\Si$ can be the underlying simplicial complex of some topological toric manifold, and the same statement holds for toric manifolds.} 

\begin{proof}
This lemma is well-known for toric manifolds and the same argument works for topological toric manifolds.  The argument is as follows.  Let $\De=(\Si,\beta)$ be a topological fan associated with a topological toric manifold.  Let $\sigma$ be a simplex in $\Si$ of maximal dimension and let $\Si'$ be the simplicial complex obtained from $\Si$ by performing the stellar subdivision of $\Si$ on $\sigma$.  We define $\beta'\colon \Si'^{(1)}\to \MR^n$ by 
\[
\beta'(i):=\begin{cases} \beta(i) \quad&\text{if $i\in \Si^{(1)}\subset \Si'^{(1)}$,}\\
\sum_{j\in \sigma^{(1)}} \beta(j) \quad &\text{if $i$ is the new vertex in $\Si'^{(1)}$,}
\end{cases}
\] 
where $\sigma^{(1)}$ denotes the set of vertices of $\sigma$.  Then $(\Si',\beta')$ is again a complete non-singular topological fan so that $\Si'$ is the underlying simplicial complex of the topological toric manifold associated with  $(\Si',\beta')$, proving the lemma. 
\end{proof}

\noindent
{\bf Lemma A.3.} 
\emph{A simplicial complex $\Si$ is the underlying simplicial complex of a topological toric manifold if and only if so is its suspension $S\Si$, and the same statement holds for a toric manifold.}   

\begin{proof}
We shall prove the lemma for topological toric manifolds.  The reader will find that the same argument will work for toric manifolds.  
 
Suppose that $\Si$ is the underlying simplicial complex of a topological toric manifold $X$.  Then the suspension $S\Si$ is the underlying simplicial complex of the product $\C P^1\times X$ which is again a topological toric manifold with the product action, proving the \lq\lq only if" part.  Conversely, suppose that the suspension $S\Si$ is the underlying simplicial complex of a topological toric manifold $Y$.  Let $x$ be a vertex of $S\Si$ created by the suspension, so the link of the $x$ is $\Si$.  The vertex $x$ corresponds to some characteristic submanifold $Y_0$ of $Y$.  The $Y_0$ is again a topological toric manifold and the underlying simplicial complex $\Si(Y_0)$ agrees with the link of the vertex $x$ in $S\Si$ which is $\Si$, proving the \lq\lq if" part.  
\end{proof}

\bigskip


\end{document}